\documentclass{amsart}

\usepackage{graphicx}
\usepackage{amsfonts}
\usepackage{amssymb}
\usepackage{amscd}
\usepackage{enumerate}
\usepackage{mathrsfs}
\usepackage{braket}

\makeatletter
\renewcommand{\section}{%
\@startsection{section}{1}%
  \z@{.7\linespacing\@plus\linespacing}{.5\linespacing}%
 {\normalfont\large\bfseries\centering}}
\makeatother

\newtheorem{theorem}{Theorem}[section]

\newtheorem{conjecture}[theorem]{Conjecture}
\newtheorem{lemma}[theorem]{Lemma}
\newtheorem{proposition}[theorem]{Proposition}

\newtheorem*{theorem*}{Theorem} 
\newtheorem*{corollary*}{Corollary}
\newtheorem*{conjecture*}{Conjecture}
\newtheorem*{lemma*}{Lemma}
\newtheorem*{proposition*}{Proposition}
\newtheorem*{problem*}{Problem}
\newtheorem*{axiom*}{Axiom}
\newtheorem*{example*}{Example}
\newtheorem*{exercise*}{Exercise}

\theoremstyle{definition}
\newtheorem{remark}[theorem]{Remark}
\newtheorem*{remark*}{Remark}

\newtheorem*{definition*}{Definition}

\renewcommand{\l}{\left}
\renewcommand{\r}{\right}
\newcommand{\eps}{\varepsilon}
\newcommand{\N}{{\mathbb N}}
\newcommand{\R}{{\mathbb R}}

\newcommand{\im}{{\rm Im}}
\newcommand{\re}{{\rm Re}}

\newcommand{\ds}{\displaystyle}

\newcommand{\del}{\partial}

\newcommand{\sgn}{{\rm sgn}}
\newcommand{\wto}{\rightharpoonup}

\def\norm[#1]{\left\Vert #1 \right\Vert}
\def\tbra[#1,#2]{\left\langle #1 , #2\right\rangle} 
\def\rbra[#1,#2]{\left( #1 , #2 \right)} 
\def\sbra[#1,#2]{\left[ #1 , #2 \right]} 

\newcommand{\scA}{{\mathscr A}}
\newcommand{\scB}{{\mathscr B}}

\newcommand{\scG}{{\mathscr G}}

\newcommand{\scK}{{\mathscr K}}

\newcommand{\scM}{{\mathscr M}}

\newcommand{\cE}{{\mathcal E}}

\newcommand{\cG}{{\mathcal G}}

\newcommand{\cI}{{\mathcal I}}

\newcommand{\cK}{{\mathcal K}}
\newcommand{\cL}{{\mathcal L}}
\newcommand{\cM}{{\mathcal M}}

\newcommand{\cP}{{\mathcal P}}

\newcommand{\cS}{{\mathcal S}}

\begin{document}

\title[Potential well theory for DNLS]{Potential well theory for the derivative nonlinear Schr\"{o}dinger equation}


\author{Masayuki Hayashi}
\address{Research Institute for Mathematical Sciences, Kyoto University, Kyoto 606-8502, Japan}
\curraddr{}
\email{hayashi@kurims.kyoto-u.ac.jp}
\thanks{}




\dedicatory{}


\begin{abstract}
We consider the following nonlinear Schr\"{o}dinger equation of derivative type:
\begin{equation}
\label{eq:1}
i \partial_t u + \partial_x^2 u +i |u|^{2} \partial_x u +b|u|^4u=0 , \quad (t,x) \in \R \times \R, \ b \in \R .
\end{equation}
If $b=0$, this equation is known as a gauge equivalent form of well-known  derivative nonlinear Schr\"{o}dinger equation (DNLS), which is mass critical and completely integrable. The equation \eqref{eq:1} can be considered as a generalized equation of DNLS while preserving mass criticality and Hamiltonian structure. For DNLS it is known that if the initial data $u_0\in H^1(\mathbb{R})$ satisfies the mass condition $\| u_0\|_{L^2}^2 <4\pi$, the corresponding solution is global and bounded. In this paper we first establish the mass condition on \eqref{eq:1} for general $b\in\R$, which is exactly corresponding to $4\pi$-mass condition for DNLS, and then characterize it from the viewpoint of potential well theory. We see that the mass threshold value gives the turning point in the structure of potential wells generated by solitons. In particular, our results for DNLS give a characterization of both $4\pi$-mass condition and algebraic solitons.


\end{abstract}

\maketitle

\tableofcontents

\numberwithin{equation}{section} 
\section{Introduction}
\subsection{Setting of the problem}
In this paper, we consider the following nonlinear Schr\"{o}dinger equation of derivative type:
\begin{equation}
 \label{eq:1.1}
i \partial_t u + \partial_x^2 u +i |u|^{2} \partial_x u +b|u|^4u=0 , \quad (t,x) \in \R \times \R, \ b \in \R .
\end{equation}
The equation \eqref{eq:1.1} is $L^2$-critical (mass critical) in the sense that the equation and $L^2$-norm are invariant under the scaling transformation 
\begin{align}
\label{eq:1.2}
u_{\lambda}(t,x)=\lambda^{\frac{1}{2}} u(\lambda^2 t,\lambda x), \quad \lambda >0. 
\end{align}
This equation has the following conserved quantities:  
\begin{align*}
\tag{Energy}
 E(u)&:=\frac{1}{2}\left\| \partial_x u  \right\|_{L^2}^2 
  - \frac{1}{4}\rbra[i|u|^2\del_xu ,u]
-\frac{b}{6}\| u\|_{L^6}^6 ,
\\
\tag{Mass}
 M(u)&:= \| u \|_{L^2}^2,
\\
\tag{Momentum}
P(u)&:=\rbra[i\partial_x u,u],
\end{align*}
where $\rbra[\cdot,\cdot]$ is an inner product defined by
\begin{align*}
\rbra[v ,w] :=\re\int_{\R} v(x)\overline{w(x)} dx\quad\text{for}~v, w\in L^2(\R ).
\end{align*}
We note that (\ref{eq:1.1}) can be rewritten as the following Hamiltonian form:
\begin{align}
\label{eq:1.3}
i \partial_t u= E ' (u).
\end{align}
When $b=0$, the equation
\begin{equation*}
\label{DNLS}
\tag{DNLS}
i \partial_t u + \partial_x^2 u +i |u|^{2} \partial_x u =0 , \quad (t,x) \in \R \times \R
\end{equation*}
is known as a standard derivative nonlinear Schr\"{o}dinger equation (DNLS). This equation has several gauge equivalent forms. If we apply the following gauge transformation to the solution of \eqref{DNLS}
\begin{align}
\label{eq:1.4}
\psi (t,x) = u (t,x)\exp\l( -\frac{i}{2} \int_{-\infty}^{x} |u (t,x)|^2 dx\r) ,
\end{align}
then $\psi$ satisfies the following equation:
\begin{equation}
\label{eq:1.5}
i \partial_t \psi + \partial_x^2 \psi +i  \partial_x( |\psi |^{2} \psi )=0 , \quad (t,x) \in \R \times \R.
\end{equation}
The equation (\ref{eq:1.5}) originally appeared in plasma physics as a model for the propagation of Alfv\'{e}n waves in magnetized plasma (see \cite{MOMT76, M76}), and it is known to be completely integrable (see \cite{KN78}). 

The equation (\ref{eq:1.1}) can be considered as a generalized equation of (\ref{DNLS}) while preserving both $L^2$-criticality and Hamiltonian structure. We note that complete integrable structure is known only for the case $b=0$. In this paper we study \eqref{eq:1.1} by variational approach not depending on complete integrability, and that enables us to treat \eqref{eq:1.1} for general $b\in\R$ in a unified way. The main aim of this paper is to investigate the structure of (\ref{eq:1.1}) from the viewpoint of potential well theory. 

\subsection{DNLS and mass critical NLS}
There is a large literature on the Cauchy problem for \eqref{DNLS}; see \cite{TF80, TF81, H93, HO92, T99, BL01, CKSTT01, CKSTT02, Wu13, Wu15, GW17, FHI17, H18, JLPS18, JLPS18a} and references therein. Here we are mainly interested in the results of energy space $H^1(\R)$. Hayashi and Ozawa \cite{HO92} proved that \eqref{DNLS} is globally well-posed in $H^1(\R)$ under the mass condition $M (u_0) < 2\pi$,  where $u_0\in H^1(\R)$ is the initial data.
The mass condition was recently improved by Wu \cite{Wu15} to $M(u_0) < 4\pi$. 
We note that the value $4\pi$ corresponds to the mass of algebraic solitons of (\ref{DNLS}) and algebraic solitons correspond to the threshold case in the existence of solitons.
We will discuss the solitons for \eqref{DNLS} in more detail later.
Fukaya, the author 
and Inui \cite{FHI17} gave a sufficient condition for global existence in $H^1(\R)$ covering Wu's result by variational approach. In particular they established the global results for the threshold case $M (u_0)=4\pi$ and $P(u_0)<0$, and for the oscillating data with arbitrarily large mass. We note that in these global results by PDE approach the class of the solution lies in $(C\cap L^{\infty})(\R, H^1(\R))$.

Recently, in \cite{JLPS18} it was proved by inverse scattering approach that \eqref{DNLS} is globally well-posed for any initial data belonging to weighted Sobolev space $H^{2,2} (\R )$, where
\begin{align*}
H^{2,2} (\R ) := \l\{ u\in H^2 (\R )~;~\braket{\cdot}^2u\in L^2(\R ) \r\}, \quad \braket{x}:=(1+x^2)^{1/2}.
\end{align*}
The class of the solution lies in $C(\R, H^{2,2}(\R))$ for the initial data $u_0\in H^{2,2} (\R )$.
This is the strong result obtained by using complete integrable structure, but the global well-posedness in the energy space $H^1(\R)$ above the mass threshold $4\pi$ is not clear yet. We note that the algebraic solitons do not belong to $H^{2,2}(\R)$, but they belong to $H^1(\R)$. This fact implies that the difference of function spaces between $H^1(\R)$ and $H^{2,2}(\R)$ is a delicate issue for \eqref{DNLS}.


\eqref{DNLS} is closely related to the focusing mass critical nonlinear Schr\"odinger equation in one space dimension.\footnote{It is also related to other mass critical dispersive equations such as the quintic Korteweg-de Vries equation and the modified Benjamin--Ono equation, but we will not discuss them further here.} Let us consider the following nonlinear Schr\"odinger equation:
\begin{align}
\label{NLS}
\tag{NLS} i\del_t v+\del_x^2 v+\frac{3}{16}|v|^4v=0, \quad (t,x) \in \R \times \R.
\end{align}
By using the following gauge transformation to the solution of \eqref{DNLS}
\begin{align}
\label{eq:1.6}
v(t,x) = u(t,x)\exp\l( \frac{i}{4} \int_{-\infty}^{x} |u(t,x)|^2 dx\r), 
\end{align}
we have another gauge equivalent form:
\begin{align}
\label{DNLS'}
\tag{DNLS$'$}
i \partial_t v + \partial_x^2 v +\frac{i}{2} |v|^{2} \del_x v -\frac{i}{2}v^2\del_x \overline{v} +\frac{3}{16}|v|^4v =0, \quad (t,x) \in \R \times \R.
\end{align}
The equations \eqref{NLS} and \eqref{DNLS'} have mass critical structure and the same conserved quantities in the forms of
\begin{align}
\tag{Energy}
&{\cE}(v)= \frac{1}{2}\| \partial_x v\|_{L^2}^2  -\frac{1}{32} \| v\|_{L^6}^6 ,\\
\tag{Mass}
&{\cM}(v) =\| v\|_{L^2}^2. 
\end{align}
Moreover, they have the same standing wave solutions as $v_{\omega}(t,x)=e^{i\omega t}Q_{\omega}(x)$, where $\omega >0$ and $Q_{\omega}>0$ is the positive solution of
\begin{align*}
-Q'' +\omega Q -\frac{3}{16}Q^5=0, ~x\in \R.
\end{align*}
From the work of Weinstein \cite{W82}, we have the following sharp Gagliardo--Nirenberg inequality
\begin{align}
\label{GN1}
\frac{1}{32}\| f\|_{L^6}^6 \leq \frac{1}{2}\l( \frac{\cM(f)}{\cM(Q_{\omega})} \r)^2\| \partial_{x}f\|_{L^2}^2
~\text{for all}~f\in H^1(\R).
\end{align}
If the initial data $v_0\in H^1(\R)$ satisfies $\cM(v_0)<\cM (Q_{\omega}) =2\pi$, by \eqref{GN1} and conservation laws of the mass and the energy, we deduce that the corresponding $H^1(\R)$-solution of \eqref{DNLS'} or \eqref{NLS} exists globally in time, and satisfies
\begin{align}
\label{eq:1.8}
\frac{1}{2}\l( 1- \l( \frac{\cM(v_0)}{\cM(Q_{\omega})} \r)^2\r) \| \del_x v(t)\|_{L^2}^2 \leq \cE(v_0)
~\text{for all } t\in\R.
\end{align}
In such a way $2\pi$-mass condition for \eqref{DNLS} was established  in \cite{HO92}. For the case of \eqref{NLS}, it is known that this mass condition is sharp, in the sense that for any $\rho \geq 2\pi$, there exists initial data $v_0 \in H^1(\R)$ such that $\cM (v_0) =\rho$ and such that the corresponding 
$H^1(\R)$-solution to (\ref{NLS}) blows up in finite time. 

\subsection{Potential well theory for NLS}
\label{sec:1.3}
Here we review the mass condition for \eqref{NLS} from the viewpoint of potential well theory. For $\omega>0$ we define the action functional by
\begin{align*}
\cS_{\omega}(\varphi) =\cE (\varphi)+\frac{\omega}{2} \cM (\varphi). 
\end{align*}
We note that $Q_{\omega}$ is a critical point of $\cS_{\omega}$, i.e., $\cS_{\omega}'(Q_{\omega}) =0$. We consider the following subsets of the energy space:
\begin{align*}
\scA_{\omega}:=& \l\{ \varphi\in H^1(\R ):  \cS_{\omega}(\varphi ) < \cS_{\omega}(Q_{\omega})\r\}, \\ 
 \scA :=&\bigcup_{ \substack{\omega>0 } } \scA_{\omega}.
\end{align*}
The set $\scA$ describes the data below the ground state in the sense of action. 
Since $\cE(Q_{\omega})=0$ and $\cM(Q_{\omega})=2\pi$, the set $\scA$ is decomposed into two disjoint sets as $\scA = \scG \cup \scB$, where
\begin{align*}
\scG &:=\l\{ \varphi \in H^1(\R) : \cM(\varphi) <2\pi \r\}, \\
\scB &:=\l\{ \varphi \in H^1(\R) :  \cM(\varphi) >2\pi , \cE (\varphi) <0 \r\}.
\end{align*}
For \eqref{NLS} the global behavior of solutions to the initial data in $\scA$ is well understood now. As can be seen above, the solution for the initial data in $\scG$ is global, and furthermore scatters both forward and backward in time (see \cite{D15}). On the other hand, the solution for the initial data in $\scB$ blows up both forward and backward in finite time (see \cite{OT91}). 

One can also give a variational characterization to the sets $\scG$ and $\scB$ as follows. We define the functional by $\cK_{\omega}(\varphi)=\l.\frac{d}{d\lambda}\cS_{\omega}(\lambda\varphi)\r|_{\lambda =1}$ for $\omega >0$, and introduce the following subsets of the energy space:
\begin{align*}
\scA_{\omega}^+ :=&\l\{ \varphi\in\scA_{\omega} :\cK_{\omega}(\varphi) \geq 0 \r\},\\
\scA_{\omega}^- :=&\l\{ \varphi\in\scA_{\omega} :\cK_{\omega}(\varphi) < 0\r\},\\
 \scA^{+} :=&\bigcup_{ \substack{\omega>0 } } \scA_{\omega}^{+}, ~
 \scA^{-} :=\bigcup_{ \substack{\omega>0 } } \scA_{\omega}^{-}.
\end{align*}
Then, one can easily prove that the sets $\scA_{\omega}^+$ and $\scA_{\omega}^-$ are invariant under the flow of \eqref{NLS},\footnote{If the initial data in $\scA_{\omega}^+$ (resp. $\scA_{\omega}^-$), then the corresponding solution of \eqref{NLS} also belongs to $\scA_{\omega}^+$ (resp. $\scA_{\omega}^-$) as long as the solutions exists.}
 and that
\begin{align}
\label{eq:1.9}
\scG =\scA^+ , ~\scB =\scA^-.
\end{align}
The key point in the proof for this claim is variational characterization of $Q_{\omega}$ on the Nehari manifold: $\medmuskip=1mu \l\{ \varphi\in H^1(\R)\setminus\{ 0\}: \cK_{\omega}(\varphi)=0\r\}$. We note that the set 
\begin{align*}
\scA^0= \l\{  \varphi\in H^1(\R) :\cM(\varphi)=2\pi, \cE(\varphi)=0\r\} 
\end{align*}
gives the boundary of both $\scG$ and $\scB$. By variational characterization of $Q_{\omega}$, this set is actually equal to the standing wave up to translations and phase shifts, i.e.,
\begin{align}
\label{eq:1.10}
\scA^0 = \l\{ e^{i\theta}Q_{\omega}(\cdot -y): \theta, y \in\R, \,\omega >0\r\} . 
\end{align}


\subsection{Two types of solitons}
\label{sec:1.4}
Despite many similarities to \eqref{NLS}, $2\pi$-mass condition is not sharp for \eqref{DNLS}.
One of this reason comes from the difference of soliton structure.
It is well-known that (\ref{DNLS}) has a two-parameter family of solitons in the form of $u_{\omega ,c}(x,t)=e^{i\omega t}\phi_{\omega ,c}(x-ct)$, where $-2\sqrt{\omega}<c\leq 2\sqrt{\omega}$, and $\phi_{\omega ,c}$ is the complex-valued solution of the equation
\begin{align*}
-\phi'' +\omega\phi +ic\phi' -i|\phi|^2\phi' =0, ~x\in \R.
\end{align*}
An interesting property for \eqref{DNLS} is that the equation has algebraic solitons (the case $c=2\sqrt{\omega}$) as well as bright solitons (the case $\omega >c^2/4$); see \cite{KN78}.  Actually we have the following explicit decay of these solitons at infinity; 
\begin{align*}
\begin{array}{ll}
\text{if}~\omega >c^2/4, &|\phi_{\omega ,c}(x)| \sim e^{-\sqrt{4\omega -c^2}|x|}\\[5pt]
\text{if}~c=2\sqrt{\omega}, &|\phi_{\omega ,c}(x)| \sim (c|x|)^{-1}
\end{array}
\text{as}~ |x| \gg 1.
\end{align*}
\indent
We note that the following curve
\begin{align}
\label{eq:1.11}
\R^+ \ni\omega \mapsto (\omega , 2s\sqrt{\omega}) \in \R^2 ~\text{for}~s\in (-1,1]
\end{align}
gives the scaling of the solitons for \eqref{DNLS}. Indeed we have the following relation
\begin{align}
\label{eq:1.12}
\phi_{\omega ,2s\sqrt{\omega}}(x) =\omega^{1/4} \phi_{1,2s} (\sqrt{\omega}x),
\end{align}
which especially implies that the mass of the solitons is invariant on this curve. From the explicit formulae of the mass of the soliton (see \cite{CO06}), we deduce that the function
\begin{align}
\label{eq:1.13}
(-1, 1]\ni s\mapsto M(\phi_{1 ,2s} ) =8\tan^{-1} \sqrt{ \frac{1+s}{1-s} } \in (0,4\pi ]
\end{align}
is strictly increasing and surjective. In particular we note that the value $4\pi$ corresponds to the mass of algebraic solitons. This property is quite different from \eqref{NLS}. By using Galilean invariance of the equation \eqref{NLS}, one can generate a two-parameter family of solitary waves
\begin{align*}
v_{\omega ,c}(t,x) = e^{i\omega t+\frac{i}{2}cx-\frac{i}{4}c^2t}Q_{\omega}(x-ct)
\end{align*}
from standing waves, but their mass is always $2\pi$ for any $\omega >0$ and $c\in\R$.

The stability of the solitons have been also studied in previous works. Colin and Ohta \cite{CO06} proved that all bright solitons for \eqref{DNLS} are orbitally stable. Their proof depends on variational arguments, which are closely related to the work of Shatah \cite{S83}. See also \cite{GW95} for partial results before \cite{CO06}. On the other hand, the orbital stability or instability for algebraic solitons is still an open problem. 

When $b>0$, the equation \eqref{eq:1.1} still has two types of solitons, but situation becomes different due to the focusing effect from the quintic term. Ohta \cite{O14} proved that for each $b>0$ there exists a unique $s^*=s^*(b) \in (0,1)$ such that the soliton $u_{\omega ,c}$ is orbitally stable if $-2\sqrt{\omega}<c<2s^*\sqrt{\omega}$, and orbitally unstable if $2s^*\sqrt{\omega}<c<2\sqrt{\omega}$. In \cite{NOW17} it was proved that the algebraic soliton $u_{\omega ,2\sqrt{\omega}}$ is orbitally unstable when $b>0$ is sufficiently small. If we observe the momentum of the solitons, the momentum is positive in the stable region, and negative in the unstable region; see Figure \ref{fig:1}. This indicates that the momentum of the soliton has an essential effect on the stability. In the borderline case $c=2s^*\sqrt{\omega}$, the momentum of the solitons is zero, and the orbital stability or instability in this case remains an open problem. 
\begin{figure}[t]
 \includegraphics[width=6cm]{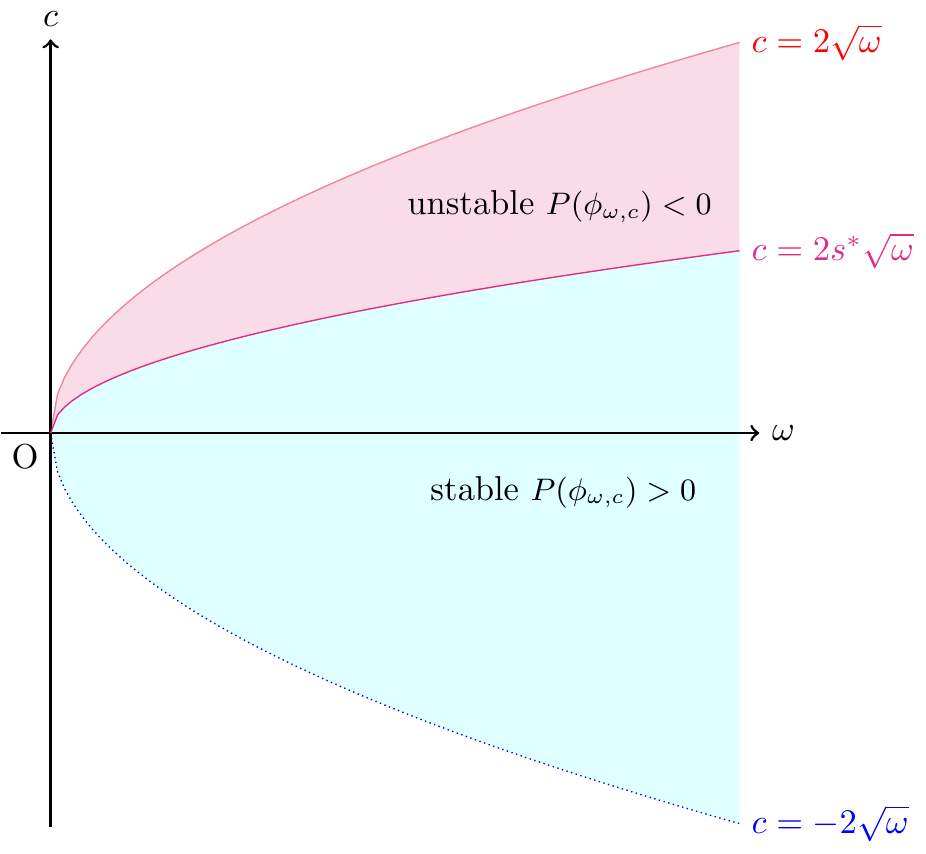}
\caption{The stable/unstable region of solitons in the case $b>0$.}
\label{fig:1}
\end{figure}

The solitons of \eqref{eq:1.1} in the case $b<0$ seem to have been little studied. Since the nonlinear term with derivative has a focusing effect, the equation \eqref{eq:1.1} still has a two-parameter family of the solitons even if the quintic term is defocusing. More precisely, we have the following result.
\begin{proposition}	
\label{prop:1.1}
Let $b<0$. The equation \eqref{eq:1.1} has a two-parameter family of solitons $u_{\omega ,c}(x,t)=e^{i\omega t}\phi_{\omega ,c}(x-ct)$ if and only if $(\omega ,c)$ satisfies
\begin{align} 
\label{eq:1.14}
\begin{array}{ll}
\ds\text{if}~b>-3/16,& \ds -2\sqrt{\omega} <c\leq 2\sqrt{\omega} ,\\[3pt]
\ds\text{if}~b\leq -3/16,& \ds -2\sqrt{\omega} <c<-2s_{\ast}\sqrt{\omega},
\end{array}
\end{align}
where $s_{\ast} :=\sqrt{ -\gamma /(1-\gamma)}$ and $\gamma :=1+\frac{16}{3}b$. 
\end{proposition}
We note that the value $b=-3/16$ gives the turning point in the structure of the solitons.
In particular algebraic solitons exist only for the case $b>-3/16$.
In the case $b\leq -3/16$ the solitons still exist, but their velocity must be negative. We note that $0\leq s_{\ast} <1$ and $s_{\ast}\uparrow 1$ as $b\downarrow -\infty$. This means that as the defocusing effect is stronger, the existence region of solitons is narrower; see Figure \ref{fig:2}. 
\begin{figure}[t]
\begin{minipage}{0.49\linewidth} 
 \includegraphics[width=\linewidth]{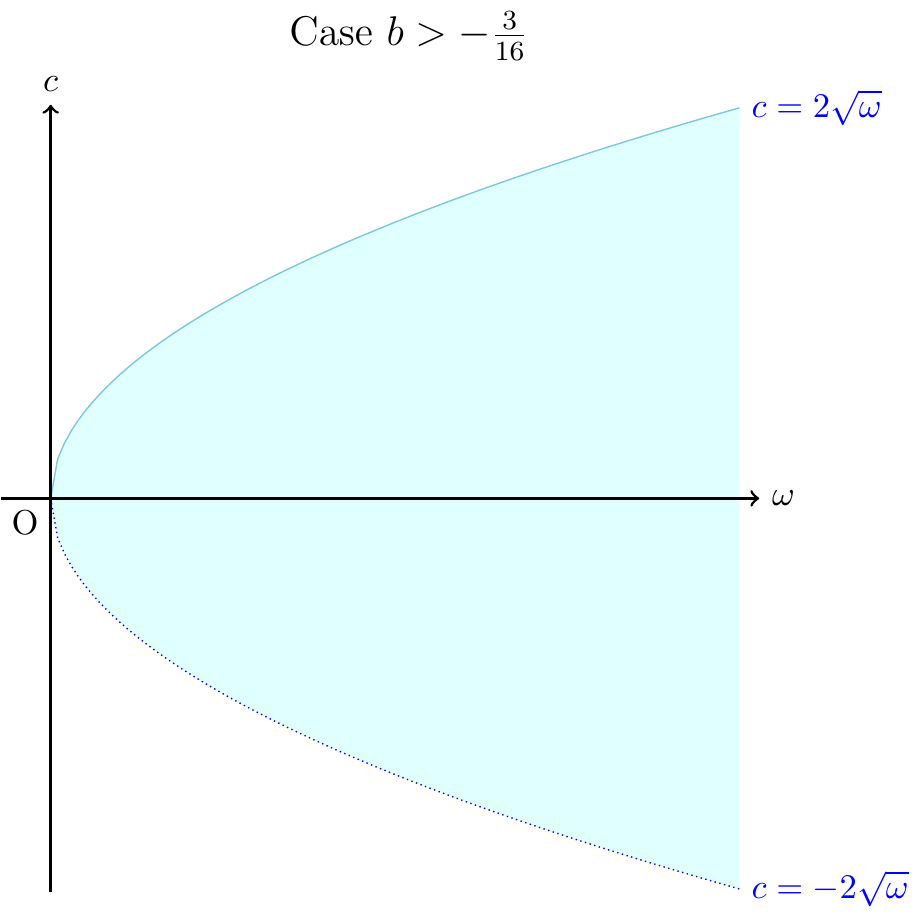}
\end{minipage}
\begin{minipage}{0.49\linewidth}
 \includegraphics[width=\linewidth]{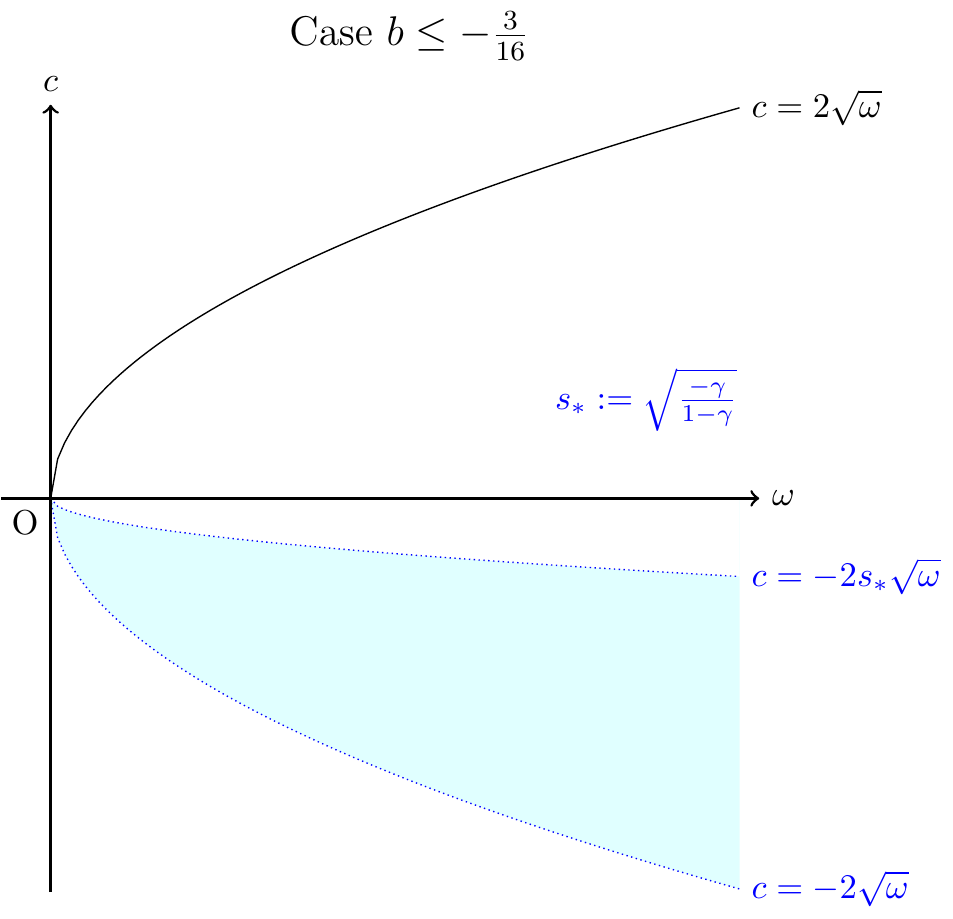}
\end{minipage}
\caption{Existence region of solitons.}
\label{fig:2}
\end{figure}

Similarly as \eqref{DNLS}, the curve \eqref{eq:1.12} gives the scaling of the solitons for the equation \eqref{eq:1.1}. For the variety of the mass we have the following result.
\begin{proposition}	
\label{prop:1.2}
Let $b\in \R$. If $b>-3/16$, the function
\begin{align*}
(-1, 1]\ni s\mapsto M(\phi_{1 ,2s} ) \in \l( 0, \frac{4\pi}{\sqrt{\gamma} } \r]
\end{align*}
is strictly increasing and surjective. Similarly, if $b\leq -3/16$, the function
\begin{align*}
(-1, -s_*)\ni s\mapsto M(\phi_{1 ,2s} ) \in (0, \infty )
\end{align*}
has the same property.
\end{proposition}
To examine the effect of the momentum is important in our analysis. We recall that the momentum of the solitons in the case $b\geq 0$ have the following property:
\begin{align*}
\begin{array}{ll}
\ds\text{if}~b=0,& P(\phi_{1,2s}) >0~\text{for}~s\in (-1,1)~\text{and}~P(\phi_{1,2})=0,\\[3pt]
\ds\text{if}~b>0,& P(\phi_{1,2s})>0~\text{for}~s\in (-1, s^*),~P(\phi_{1,2s^*})=0\\[3pt]
&\text{and}~P(\phi_{1,2s})<0~\text{for}~s\in(s^{*},1].
\end{array}
\end{align*}
One can prove that $s^*(b)\to 1$ as $b\downarrow 0$ (see Remark \ref{rem:2.8}). In this sense we set $s^*(0):=1$. We note that the value $s^*$ is characterized by
\begin{align}
\label{eq:1.15}
P(\phi_{1,2s^*(b)}) =0 \quad \text{for all}~b\geq 0.
\end{align}
For the momentum of the solitons in the case $b<0$, we have the following result.
\begin{proposition}
\label{prop:1.3}
Let $b<0$. The momentum of all solitons for the equation \eqref{eq:1.1} is positive.
\end{proposition}
In Section \ref{sec:2} we study the solitons of \eqref{eq:1.1} in more detail, and give a proof for these propositions.


\subsection{Main results}
First we establish the mass condition for the equation (\ref{eq:1.1}). The local well-posedness in the energy space $H^1(\R)$ was obtained in \cite{HO94a, Oz96}. In \cite{Oz96} it was proved that
\eqref{eq:1.1} was globally well-posed in $H^1(\R)$ under the mass condition
\begin{align}
\label{eq:1.16}
\begin{array}{ll}
\ds M(u_0) < \frac{2\pi}{\sqrt{\gamma}} &\text{if}~ b>0, \\[9pt]
\ds M(u_0) <2\pi &\text{if}~ b\leq 0,
\end{array}
\end{align}
where we recall that $\gamma =1+\frac{16}{3}b$. 
This result is considered as a natural extension of $2\pi$-mass condition for \eqref{DNLS}.\footnote{Actually more general equation including \eqref{eq:1.1} is studied in \cite{Oz96}. }
From the following energy form
\begin{align}
\label{eq:1.17}
E\l( \cG_{-1/4}(u) \r) =\frac{1}{2}\| \del_x u \|_{L^2}^2 -\frac{\gamma}{32}\| u\|_{L^6}^6,
\end{align}
where 
\begin{align*}
\cG_{a}(u)(t,x):=\exp\l( ia\int_{-\infty}^{x}|u(t,y)|^2 dy\r)u(t,x) \quad\text{for}~a\in\R, 
\end{align*}
the mass condition seems to be necessary when $b>-3/16$. By using the sharp Gagliardo--Nirenberg inequality \eqref{GN1} and the conservation laws of mass and energy, $\frac{2\pi}{\sqrt{\gamma}}$-mass condition is obtained when $b>-3/16$. We note that the value $\frac{2\pi}{\sqrt{\gamma}}$ corresponds to the mass of the standing waves of \eqref{eq:1.1}, i.e., $\frac{2\pi}{\sqrt{\gamma}} =M(\phi_{\omega,0})$. 

Our first result gives the improvement of the mass condition in previous works.
 \begin{theorem}
\label{thm:1.4}
Let $u_0\in H^1(\R)$ satisfy each of the following two cases\textup{:}
\begin{enumerate}[\rm (i)]
\setlength{\itemsep}{3pt}
\item If $b>0$, $M(u_0) < M(\phi_{1,2s^*})$, or  $M(u_0) =M(\phi_{1,2s^*})$ and $P(u_0)<0$.
\item If $-3/16<b\leq 0$, $M(u_0) < \frac{4\pi}{\gamma^{3/2}}$, or  $M(u_0) =\frac{4\pi}{\gamma^{3/2}}$ and $P(u_0)<0$.
\end{enumerate}
Then the $H^1(\R)$-solution $u$ of \eqref{eq:1.1} with $u(0)=u_0$ exists globally both forward and backward in time. Moreover we have
\begin{align*}
\sup_{t\in\R} \| u(t)\|_{H^1}\leq C(\| u_0\|_{H^1})<\infty .
\end{align*}
\end{theorem}
\begin{remark}
\label{rem:1.5}
When $b\leq -3/16$, the equation \eqref{eq:1.1} is globally well-posed for any initial data $u_0\in H^1(\R)$. In particular the global result in the case $b=-3/16$ is compatible with Theorem \ref{thm:1.4}, since $\frac{4\pi}{\gamma^{3/2}}\uparrow\infty$ as $b\downarrow -3/16$.
\end{remark}
We recall that if $b>0$ the soliton $\phi_{1,2s^*}$ corresponds to borderline case in the stable/unstable region of solitons as in Figure \ref{fig:1}. By Proposition \ref{prop:1.2}, we have the following relation
\begin{align*}
 \frac{2\pi}{\sqrt{\gamma}} =M(\phi_{1,0})
 <M(\phi_{1,2s^*}) < M(\phi_{1,2}) = \frac{4\pi}{\sqrt{\gamma}},
\end{align*}
which implies that our mass condition improves the one \eqref{eq:1.16}. We note that
\begin{align*}
M (\phi_{1,2s^*})  \to 4\pi ~\text{as}~b\downarrow 0,
\end{align*}
which follows from the claim that $s^*(b)\to 1$ as $b\downarrow 0$. This means that the mass condition in Theorem \ref{thm:1.4} is compatible with $4\pi$-mass condition for (\ref{DNLS}).

The mass condition in the case $-3/16<b<0$ is more interesting. Since $0<\gamma <1$ in this case, the value $\frac{4\pi}{\gamma^{3/2}}$ is greater than $4\pi$. This means that $4\pi$-mass condition for (\ref{DNLS}) is improved due to the defocusing effect from the quintic term. Moreover, the value $\frac{4\pi}{\gamma^{3/2}}$ is even greater than the mass of algebraic solitons. Indeed, we have the following relation:
\begin{align*}
M (\phi_{1,2})  = \frac{4\pi}{\sqrt{\gamma}}<
\frac{4\pi}{\gamma^{3/2}} =M(\phi_{1,2})+P(\phi_{1,2}),
\end{align*}
which indicates that positive momentum of algebraic solitons boosts the threshold value. 

Our next result is a global result for large data. If we consider sufficiently oscillating data, we obtain the global result for arbitrarily large mass:     
\begin{theorem}
\label{thm:1.6}
Let $b>-3/16$. Given $\psi \in H^1(\R )$, and set the initial data as $u_{0,\mu}=e^{i\mu x}\psi$. Then, there exists $\mu_0=\mu_0(\psi ) >0$ such that if $\mu\geq\mu_0$, then the $H^1(\R)$-solution $u_{\mu}$ of \eqref{eq:1.1} with $u_{\mu}(0)=u_{0,\mu}$ exists globally both forward and backward in time. Moreover we have
\begin{align*}
\sup_{t\in\R} \| u_{\mu}(t)\|_{H^1}\leq C(\| u_{0,\mu}\|_{H^1})<\infty .
\end{align*}
\end{theorem}
This global result was first discovered in \cite{FHI17} for \eqref{DNLS}.\footnote{As seen in \cite{FHI17}, this global result still holds for the generalized derivative nonlinear Schr\"{o}dinger equation in $L^2$-supercritical setting. }
It is worthwhile to compare the global results for the quadratic oscillating data in \eqref{NLS}. Cazenave and Weissler \cite{CW92} established global existence for oscillating data as follows: Given $\psi \in H^{1,1}(\R)$ which is defined by
\begin{align*}
H^{1,1} (\R ) := \l\{ u\in H^1 (\R )~;~\braket{\cdot}u\in L^2(\R ) \r\}.
\end{align*}
Set the initial data as $u_{0,\beta} :=e^{i\frac{\beta |x|^2}{4}}\psi$. Then, there exists $\beta_0=\beta_0(\psi ) >0$ such that if $\beta\geq\beta_0$, the corresponding solution $u_{\beta}$ for \eqref{NLS} satisfies $C((0, \infty), H^{1,1}(\R))$.\footnote{The solution $u_{\beta}$ also satisfies $u_{\beta}\in L^{\infty}((0, \infty), H^1 (\R))$.} 

We note that the quadratic oscillating data only yields global solutions forward in time. In general the solution $u_{\beta}$ may blow up in a finite negative time (see \cite[Remark 6.5.9]{C03}).   
The other important difference is that the oscillating factor in Theorem \ref{thm:1.6} comes from the change of the momentum, but on the other hand, the quadratic oscillating factor in \cite{CW92} comes from pseudo-conformal transformation. We note that \eqref{eq:1.1} has no Galilean or pseudo-conformal invariance. In particular, due to the lack of Galilean invariance it is reasonable to consider that the momentum of initial data essentially influences global properties of the solution to (\ref{eq:1.1}).

The proofs of Theorem \ref{thm:1.4} and Theorem \ref{thm:1.6} are done by adapting variational arguments developed in the work \cite{FHI17}, and actually obtained from a more general result (Theorem \ref{thm:1.7} below). The key point in our approach is to give a variational characterization of the solitons on the Nehari manifold with respect to the action functional. However, for the equation \eqref{eq:1.1} when $b<0$, the quintic term $b|u|^4u$ becomes an obstacle to characterize the solitons.
To overcome that we consider the following gauge equivalent form:
\begin{align}
\label{ME}
\tag{1.1$'$}
i\del_t v+\del_x^2 v+\frac{i}{2}|v|^2\del_x v-\frac{i}{2}v^2\del_x\overline{v}+\frac{3}{16}\gamma |v|^4v=0,\quad (t,x) \in \R \times \R.
\end{align}
We note that \eqref{ME} is transformed from \eqref{eq:1.1} through the gauge transformation $v=\cG_{1/4}(u)$. When $b=0$ this equation is nothing but \eqref{DNLS'}. The equation \eqref{ME} has the following conserved quantities and solitons: 
\begin{align*}
\tag{Energy}
 \cE (v) &:=\frac{1}{2} \| \del_x v\|_{L^2}^2-\frac{\gamma}{32} \| v\|_{L^6}^6 ,
\\
\tag{Mass}
 \cM (v)&:=\| v \|_{L^2}^2,
\\
\tag{Momentum}
\cP (v)&:=\rbra[i\del_x v,v] +\frac{1}{4} \| v\|_{L^4}^4,\\
\tag{Soliton}
v_{\omega ,c}(t ,x) &:=\cG_{1/4}(u_{\omega ,c})(t,x)=e^{i\omega t}\varphi_{\omega ,c}(x-ct).
\end{align*}
We note that global well-posedness in $H^1(\R)$ for the equations \eqref{eq:1.1} and \eqref{ME} is equivalent since $u\mapsto\scG_{1/4}(u)$ is locally Lipschitz continuous on $H^1(\R)$.
 
From the energy formula of \eqref{ME} one can characterize solitons on the Nehari manifold if $b\geq -3/16$ (see Proposition \ref{prop:4.1}). Based on this variational characterization we formulate potential well theory with two parameters which is related to the classical work of Payne and Sattinger \cite{PS75}. We see that a two-parameter family of potential wells has rich and complex structure compared with the one-parameter one as in Section \ref{sec:1.3}. 

To state our main results, we prepare some notations. We define the action functional by
\begin{align*}
\cS_{\omega ,c}(\varphi ) :=\cE (\varphi)+\frac{\omega}{2} \cM (\varphi) +\frac{c}{2}\cP (\varphi).
\end{align*}
We note that $\varphi_{\omega, c}$ is a critical point of $\cS_{\omega,c}$, i.e., $\cS_{\omega ,c}'(\varphi_{\omega ,c}) =0$. We also define the functional by $\cK_{\omega ,c}(\varphi):=\l.\frac{d}{d\lambda}\cS_{\omega ,c}(\lambda\varphi)\r|_{\lambda =1}$. 
Similarly as in the case of \eqref{NLS}, we consider the following subsets of the energy space: 
\begin{align*}
\scA_{\omega ,c}:=& \l\{ \varphi\in H^1(\R ):  \cS_{\omega, c}(\varphi ) < \cS_{\omega ,c}(\varphi_{\omega, c})\r\}, \\ 
\scA_{\omega, c}^+ :=&\l\{ \varphi\in\scA_{\omega, c} :\cK_{\omega, c}(\varphi) \geq 0 \r\},\\
\scA_{\omega, c}^- :=&\l\{ \varphi\in\scA_{\omega, c} :\cK_{\omega ,c}(\varphi) < 0\r\}.
\end{align*}
Here we introduce the potential well along the scaling curve:
\begin{align*}
\scA_{s}:= \bigcup_{ \substack{\omega>0 } } \scA_{\omega ,2s\sqrt{\omega}},~
\scA_{s}^{\pm}:= \bigcup_{ \substack{\omega>0 } } \scA_{\omega ,2s\sqrt{\omega}}^{\pm}
\quad\text{for}~s\in (-1,1].
\end{align*}
We define the mass threshold value in Theorem \ref{thm:1.4} as
\begin{align}
\label{eq:1.18}
M^* =M^*(b):=
\l\{
\begin{array}{ll}
M(\phi_{1,2s^*(b)}) &\text{if}~b\geq 0,\\[3pt]
M(\phi_{1,2})+P(\phi_{1,2})  &\text{if}~-3/16<b\leq 0.
\end{array}
\r.
\end{align}
We note that $M^*(0)$ is well defined since 
\begin{align*}
M(\phi_{1,2s^*(0)})=M(\phi_{1,2})~\text{and}~P(\phi_{1,2})=0 \quad\text{when}~b=0.
\end{align*}


Our main result in this paper is the following classification of a two-parameter family of potential wells which covers Theorem \ref{thm:1.4} and Theorem \ref{thm:1.6}.
\begin{theorem}
\label{thm:1.7}
Let $b>-3/16$ and let $(\omega ,c)$ satisfy $-2\sqrt{\omega}<c\leq 2\sqrt{\omega}$. Then, each of $\scA_{\omega ,c}^{+}$ and 
$\scA_{\omega ,c}^{-}$ is invariant under the flow of \eqref{ME}. If $v_0 \in\scA_{\omega ,c}^+$, then the $H^1(\R)$-solution $v$ of \eqref{ME} with $v(0)=v_0$ exists globally both forward and backward in time, and satisfies the following uniform estimate\textup{:}
\begin{align}
\label{eq:1.19}
\| \del_x v\|_{L^{\infty}(\R ,L^2)}^2 \leq 8\cS_{\omega ,c}(v_0)+\frac{c^2}{2}\cM (v_0).
\end{align}
Moreover the following statements hold\textup{:}
\begin{enumerate}[\rm (i)]
\setlength{\itemsep}{3pt}
\item For each $s\in (-1,1]$, $\scA^+_{s}$ and $\scA^-_s$ have no elements in common on the set $\{ \varphi\in H^1(\R) : \cM(\varphi ) \geq M^*\}$. 

\item  If $\cM (\varphi)<M^*$, or $\cM (\varphi)=M^*$ and $\cP (\varphi)<0$, then $\varphi\in\scA^+_{s^*}$ if $b\geq 0$, or $\varphi\in\scA^+_{1}$ if $-3/16< b\leq 0$.

\item For given $\psi\in H^1(\R)\setminus\{ 0\}$ the following properties hold\textup{:}
\begin{enumerate}[\rm (a)]
\item There exists $\mu_0 =\mu_0 (\psi)>0$ such that if $\mu\geq\mu_0$, then $e^{i\mu x}\psi\in \scA^+_1$.

\item There exist $\eps\in (0,1)$ and large $\mu >0$ such that $e^{-i(1-\eps )\mu x}\psi\in \scA^-_{-(1-\eps)}$, where $\eps$ and $\mu$ depend on $\psi$.
\end{enumerate}

\item Assume $\cE(\varphi)<0$.
 Then $\varphi\in\bigcap_{-1<s\leq 1}\scA^-_s$. In particular, if $M(\varphi)\geq M^*$, then $\varphi\not\in\bigcup_{-1< s\leq 1}\scA_s^+$.\footnote{The negative energy is possible only when $\cM(\varphi)> \frac{2\pi}{\sqrt{\gamma}}$. Note that the following case
\begin{align*}
\cE (\varphi) <0, \cM(\varphi) =M^*~\text{and}~\cP (\varphi)\leq 0
\end{align*} 
does not occur from the assertion (vi).}

\item Assume $\cE(\varphi) \geq 0$ and $\cM (\varphi) \geq M^*$. 
If $\cP(\varphi) \geq 0\,(\text{resp.}\,\cP(\varphi)\leq 0)$, then $\varphi\not\in\bigcup_{0\leq s\leq 1}\scA_s\,(\text{resp.}\,\varphi\not\in\bigcup_{-1<s\leq 0}\scA_s)$. In particular, if $\cP(\varphi)=0$, then $\varphi\not\in\bigcup_{-1< s\leq 1}\scA_s$.

\item Assume $\cM (\varphi)=M^*$. Then the following properties hold\textup{:}
\begin{enumerate}[\rm(a)]
\item When $b\geq 0$, $\cE(\varphi)=\cP(\varphi)=0$ if and only if there exist $\theta , y\in\R$ and $\omega >0$ such that $\varphi =e^{i\theta}\varphi_{\omega ,2s^*\sqrt{\omega} }(\cdot -y)$. Moreover, there exists no $\varphi\in H^1(\R)$ such that $\cE(\varphi)<0$ and $\cP(\varphi)\leq 0$, or $\cE(\varphi)\leq 0$ and $\cP(\varphi) <0$.

\item When $-3/16 <b<0$, there exists no $\varphi\in H^1(\R)$ such that $\cE(\varphi)\leq 0$ and $\cP(\varphi)\leq 0$.

\end{enumerate}
\end{enumerate}
\end{theorem}
\begin{remark}
\label{rem:1.8}
Concerning the assertion (i), we recall that $\scA^+$ and $\scA^-$ for \eqref{NLS} in Section \ref{sec:1.3} 
are mutually disjoint.
However, in general $\scA_s^+$ and $\scA_s^-$ have elements in common on the set $\{ \varphi\in H^1(\R) : \cM(\varphi ) < M^*\}$. For example it follows from the assertions (ii) and (iv) that
\begin{align*}
\cE(\varphi)<0, ~\frac{2\pi}{\sqrt{\gamma}}<\cM(\varphi)<M^* 
\Longrightarrow 
\begin{array}{ll}
\varphi\in \scA_{s^*}^+ \cap \scA_{s^*}^- &\text{if}~b\geq0, \\
\varphi\in \scA_{1}^+ \cap \scA_{1}^- &\text{if}~-\frac{3}{16}<b\leq 0.
\end{array}
\end{align*}
 This gives a notable property of two-parameter family of potential wells, which is also closely related to the stability of solitons (see \cite{H19}). We note that the interval $(0,M^*(b))$ for $b\geq 0$ corresponds to the range of the mass of stable solitons.
\end{remark}
\begin{remark}
\label{rem:1.9}
For given $\psi\in H^1(\R)$ and $c\in\R$, we have
\begin{align}
\label{eq:1.20}
\cE (e^{icx}\psi)\sim c^2 , ~ \cP (e^{icx}\psi)\sim -c\quad\text{as}~|c|\gg 1.
\end{align}
In particular one can see that oscillating factor in Theorem \ref{thm:1.6} causes to change the momentum to the negative direction. The assertion (iii-b) gives the counterpart of this global result, and implies that the oscillating direction is essential for generating global and bounded solutions.
\end{remark}
\begin{remark}
\label{rem:1.10}
When $b=-3/16$, one can still prove that $\scA_{\omega ,c}^{\pm}$ is invariant under the flow, and that for $v_0\in\scA^+_{\omega ,c}$ the corresponding solution satisfies the uniform estimate \eqref{eq:1.19}. Moreover we have the following claim (Proposition \ref{prop:5.2}):
\begin{align*}
\bigcup_{-1< s<0}\scA_s =\bigcup_{-1< s<0}\scA_s^+=H^1(\R).
\end{align*}
This gives the characterization by potential well theory to the global result in the case $b=-3/16$.
\end{remark}
Theorem \ref{thm:1.7} is the first classification theorem for a two-parameter family of potential wells. For \eqref{DNLS} the relation between $4\pi$-mass condition and $\scA^+_{\omega ,c}$ was first pointed out in \cite{FHI17}, but in the present paper we give a characterization for $\scA^-_{\omega ,c}$ as well as $\scA^+_{\omega ,c}$. Adopting a family of potential wells along the scaling curve is a new idea, which is useful to examine the properties of potential wells.

The assertions (ii) and (iii-a) give a representation by potential wells to the global results in Theorem \ref{thm:1.4} and Theorem \ref{thm:1.6}. It follows from Proposition \ref{prop:1.2} that $\scA_s^+$ contains solitons with arbitrarily small mass,\footnote{Furthermore, one can prove that $\| \phi_{1,2s}\|_{H^1}\to 0$ as $s\downarrow -1$, which implies that \eqref{eq:1.1} has the solitons with arbitrarily small $H^1(\R)$-norm.} which implies that the solutions for the data in $\scA_s^+$ do not scatter in general.
Also, we note that the equation \eqref{eq:1.1} corresponds to the long range scattering, and it is known that modified scattering occurs for small data in weighted Sobolev spaces (see \cite{HO94modi, Oz96, GHLN13}). These properties give quite different situation from the set $\scA^+$ for \eqref{NLS}. 
For \eqref{DNLS} it was proved in \cite{JLPS18a} that the soliton resolution holds for generic data in $H^{2,2}(\R)$, but the global dynamics in the energy space is still far from clear.   

The assertions (iv) and (v) show the optimality of the mass threshold value $M^*$ in the sense that for any $\rho\geq M^*$ there exists $\varphi\in H^1(\R)$ such that $\cM (\varphi)=\rho$ and $\varphi\not\in \scA^+_{s}$ for any $s\in (-1,1]$. The set of the data satisfying
\begin{align}
\label{eq:1.21}
\cE(\varphi)<0, \cM(\varphi)>M^* ~\text{or}~\cE(\varphi)<0, \cM(\varphi)=M^*, \cP (\varphi)>0 
\end{align}
is an important subset contained in $\bigcap_{-1<s\leq 1}\scA^-_s$, and has analogy with the set $\scB$ for \eqref{NLS}.
The set of the data satisfying
\begin{align}
\label{eq:1.22}
\cE(\varphi)\geq 0, \cM(\varphi) \geq M^*, \cP(\varphi)=0 
\end{align}
gives a subset of the complement of $\scA_s$ for any $s\in (-1,1]$. If we replace $\cP(\varphi)=0$ by $\cP(\varphi)\neq 0$, then this inclusion does not hold.  Indeed, it follows from \eqref{eq:1.20} and the assertion (iii) that the oscillating data $e^{icx}\psi$ for large $|c|>0$ gives the counterexample.

The assertion (vi) gives some constraint condition on $\cM(\varphi )= M^*$. When $b\geq 0$ the following set
\begin{align*}
\scB_0 :=\l\{ \varphi\in H^1(\R) : \scM (\varphi)=M^*, \,\cE (\varphi)=\cP (\varphi)=0\r\}
\end{align*}
gives the boundary of both $\scA_{s^*}^+$ and $\scA_{s^*}^-$. From (vi-a) the set $\scB_0$ corresponds to the borderline solitons with respect to stability/instability, i.e.,
\begin{align}
\label{eq:1.23}
\scB_0 = \l\{ e^{i\theta}\varphi_{\omega ,2s^*\sqrt{\omega}}(\cdot -y): \theta, y \in\R, \,\omega >0\r\},
\end{align}
which gives analogy with the relation \eqref{eq:1.10} for \eqref{NLS}. On the other hand, when $-3/16<b<0$ the set $\scB_{0}$ is empty.

From Theorem \ref{thm:1.7} we see that the mass threshold value $M^*$ gives the turning point in the structure of potential wells. In this sense $M^*$ corresponds to the value $2\pi$ for \eqref{NLS}.
Therefore, taking into account the mass critical structure of the equation, we conjecture that the mass condition in Theorem \ref{thm:1.4} is sharp. To state more precisely, let us say that $(GE)(u_0)$ holds for $u_0\in H^1(\R)$ if  the $H^1(\R)$-solution $u$ of \eqref{eq:1.1} with $u(0)=u_0$ is global both forward and backward in time, and uniformly bounded in $H^1(\R)$, i.e., $u\in (C\cap L^{\infty})(\R ,H^1(\R))$. We define the positive value $m^*$ by
\begin{align*}
m^* := \sup \l\{ m>0 : \forall u_0 \in H^1(\R), M(u_0)<m \Rightarrow (GE)(u_0)~\text{holds} \r\}.
\end{align*}
Then, our conjecture is organized as follows:
\begin{conjecture}
\label{conj:1.11}
Let $b>-3/16$. Then $m^*=M^*$.
\end{conjecture}
Theorem \ref{thm:1.4} implies that $m^*\geq M^*$. Although for \eqref{DNLS} global existence was proved for any initial data in $H^{2,2}(\R)$ in \cite{JLPS18}, we note that their results do not imply nonexistence of infinite time blow-up solutions, i.e., the $H^1(\R)$-norm of the solution may be unbounded in time. Also, existence/nonexistence of finite time blow-up solutions in $H^1(\R)$ is still an open problem. It is known that finite time blow-up occurs for \eqref{DNLS} on a bounded interval or on the half line, with Dirichlet boundary condition (see \cite{Tan04, Wu13}). The data satisfying the condition \eqref{eq:1.21} is a good candidate generating singular solutions. From Theorem \ref{thm:1.7} and analogy with \eqref{NLS}, one can say that this condition gives certain obstruction for generating global and bounded solutions. 

Related to Conjecture \ref{conj:1.11}, one can prove the following blow-up criterion:
\begin{theorem}
\label{thm:1.12}
Let the initial data $u_0\in H^1(\R)$ satisfy $M(u_0)=4\pi$. Suppose that the corresponding solution $u$ of \eqref{DNLS} blows up in time $T^*\in (0,\infty]$.\footnote{We say that the solution $u$ blows up in infinite time if $\lim_{t\to\infty}\| \del_x u(t)\|_{L^2}=\infty$.} Then, there exist functions $\theta (t)\in \R$ and $y(t)\in\R$ such that 
\begin{align}
\label{eq:1.24}
u(t)-\frac{e^{i\theta (t)}}{\lambda(t)^{1/2}}\phi_{1,2}\l( \frac{x-y(t)}{\lambda (t)}\r)  \to 0~\text{in}~H^1(\R)~\text{as}~t\to T^*,
\end{align}
where $\lambda (t):=\| \del_x\phi_{1,2}\|_{L^2} / \| \del_x u(t) \|_{L^2}$.
\end{theorem}
This result is analogous to the one obtained by Weinstein \cite{W86} for \eqref{NLS}.
The similar result of Theorem \ref{thm:1.12} was first obtained in \cite{KW18}.  
Here we give a simple alternative proof by using characterization on the Nehari manifold which is related to concentration compactness arguments in \cite{W86}. 

\subsection{Organization of the paper}
The rest of this paper is organized as follows. In Section \ref{sec:2} we study the solitons of the equation \eqref{eq:1.1} and calculate the conserved quantities of them. In Section \ref{sec:3} we review the gauge transformation and the local well-posedness theory in the energy space. In Section \ref{sec:4} we give a variational characterization of two types of solitons in a unified way. In Section \ref{sec:5} we establish potential well theory for \eqref{ME} by applying the variational characterization, and give a proof of Theorem \ref{thm:1.7}. In Section \ref{sec:6} we organize 
potential well theory for \eqref{DNLS}, and give a proof of Theorem \ref{thm:1.12}. 
\section{Solitons and conserved quantities}
\label{sec:2}
\subsection{A two-parameter family of solitons}
\label{sec:2.1}
Here we formulate the solitons of (\ref{eq:1.1}) following \cite{CO06}.
Consider solutions of (\ref{eq:1.1}) of the form
\begin{align}
\label{eq:2.1}
u_{\omega,c}(t,x)=e^{i\omega t} \phi_{\omega,c}(x-ct)
\end{align}
for $(\omega,c) \in\R^2$, and assume that $\phi_{\omega ,c} \in H^1(\R)$. It is clear that $\phi_{\omega ,c}$ must satisfy the following equation:
\begin{align}
\label{eq:2.2}
-\phi ''+ \omega \phi +ic \phi ' - i|\phi|^{2} \phi '-b|\phi |^4\phi =0, \quad x \in\R.
\end{align}
We note that the equation (\ref{eq:2.2}) can be rewritten as $S_{\omega ,c}' (\phi) =0$, where
\begin{align}
\label{eq:2.3}
S_{\omega ,c}(\phi ):=E(\phi )+\frac{\omega}{2} M(\phi )+\frac{c}{2}P(\phi ).
\end{align}
Applying the following gauge transformation to $\phi_{\omega ,c}$
\begin{align}
\label{eq:2.4}
\phi_{\omega,c}(x) 
&= \Phi_{\omega,c}(x) \exp\l( i\frac{c}{2}x - \frac{i}{4} \int_{-\infty}^{x} \l|\Phi_{\omega,c}(y)\r|^2 dy \r),
\end{align}
then $\Phi_{\omega ,c}$ satisfies the following equation:
\begin{align}
\label{eq:2.5}
- \Phi ''+ \l(\omega- \frac{c^2}{4}\r) \Phi +\frac{1}{2}\im\l( \overline{\Phi}\Phi'\r) \Phi+ \frac{c}{2} |\Phi|^{2} \Phi - \frac{3}{16}\gamma |\Phi|^{4}\Phi =0, \quad x\in \R,
\end{align}
where $\gamma = 1+\frac{16}{3}b$. From $\Phi_{\omega ,c}\in H^1(\R)$ and the equation \eqref{eq:2.5}, one can show that $\im\l( \overline{\Phi_{\omega ,c}}\Phi_{\omega ,c}'\r)=0$ (see \cite[Lemma 2]{CO06}). Therefore $\Phi_{\omega ,c}$ satisfies the following elliptic equation with double power nonlinearity:
\begin{align}
\label{eq:2.6}
- \Phi ''+ \l(\omega- \frac{c^2}{4}\r) \Phi + \frac{c}{2} |\Phi|^{2} \Phi - \frac{3}{16}\gamma |\Phi|^{4}\Phi =0, \quad x\in \R.
\end{align}
The positive radial (even) solution of (\ref{eq:2.6}) is explicitly obtained as follows; 
if $\gamma >0$, 
\begin{align}
\label{eq:2.7}
\Phi_{\omega,c}^2(x) &=
\l\{
\begin{array}{ll}
\thickmuskip=0mu 
\medmuskip=1mu 
\thinmuskip=0mu 
\ds
 \frac{ 2(4\omega - c^2) }{\sqrt{c^2+\gamma (4\omega -c^2)} \cosh ( \sqrt{4\omega- c^2}x)-c }
& 
\thickmuskip=2mu 
\ds \text{if}~-2\sqrt{\omega}<c<2\sqrt{\omega},
\\
& 
\\ 
\medmuskip=1mu
\ds
\frac{4c}{(cx)^2+\gamma} 
& 
\thickmuskip=3.2mu 
\ds \text{if}~c=2\sqrt{\omega},
\end{array}
\r.
\end{align}
or if $\gamma \leq 0$,
\begin{align}
\label{eq:2.8}
\Phi_{\omega,c}^2(x) &=
\begin{array}{ll}\ds
\medmuskip=1mu 
\thinmuskip=0mu 
\frac{ 2(4\omega - c^2) }{\sqrt{c^2+\gamma (4\omega -c^2)} \cosh ( \sqrt{4\omega- c^2}x)-c } 
& 
\thickmuskip=2mu 
\medmuskip=1mu 
\ds \text{if}~-2\sqrt{\omega} <c<-2s_{\ast}\sqrt{\omega}, 
\end{array}
\end{align}
where $s_{\ast} =s_{\ast}(\gamma)=\sqrt{ -\gamma /(1-\gamma)}$. We note that the condition
\begin{align} 
\label{eq:2.9}
\begin{array}{ll}
\ds\text{if}~ \gamma >0\Leftrightarrow b>-3/16,& \ds -2\sqrt{\omega} <c\leq 2\sqrt{\omega} ,\\[7pt]
\ds\text{if}~ \gamma \leq 0\Leftrightarrow b\leq -3/16,& \ds -2\sqrt{\omega} <c<-2s_{\ast}\sqrt{\omega} 
\end{array}
\end{align}
is a necessary and sufficient condition for the existence of non-trivial solutions of (\ref{eq:2.6}) vanishing at infinity; see \cite{BeL83}. 
From (\ref{eq:2.1}), (\ref{eq:2.4}), (\ref{eq:2.7}) and (\ref{eq:2.8}), we obtain the following explicit formulae of solitons:
\begin{align}
\label{eq:2.10}
u_{\omega ,c}(t,x)=e^{i\omega t +i\frac{c}{2}(x-ct)-\frac{i}{4}\int_{-\infty}^{x-ct}|\Phi_{\omega ,c}(y)|^2 dy}\Phi_{\omega ,c}(x-ct).
\end{align} 

\subsection{Mass of the solitons} 
\label{sec:2.2}%
In this subsection we calculate the mass of the solitons. First we prepare the following elementary integration formulae:
\begin{lemma}
\label{lem:2.1}
Let $-1<\alpha $. Then we have
\begin{align}
\label{eq:2.11}
\int_{-\infty}^{\infty} \frac{dy}{\cosh y+\alpha } =\l\{
\begin{array}{lll}
\ds \frac{4}{ \sqrt{1-\alpha^2} }\tan^{-1} \sqrt{ \frac{1-\alpha}{1+\alpha} } &\ds \text{\rm if} &|\alpha |<1,\\[8pt]
\ds \quad 2 &\ds \text{\rm if} &\alpha =1, \\[2pt]
\ds \frac{2}{ \sqrt{\alpha^2 -1} } \log \l( \alpha +\sqrt{\alpha^2 -1}\r) &\ds \text{\rm if} &\alpha >1.
\end{array}
\r.
\end{align}
\end{lemma}
\begin{proof}
See the formula 3.513, 2 in \cite{GR07}.
\end{proof}
By using Lemma \ref{lem:2.1}, we have the following proposition.
\begin{proposition}
\label{prop:2.2}
Let $(\omega ,c)$ and $\gamma$ satisfy \eqref{eq:2.9}. Then the following properties hold\textup{:}
\begin{enumerate}[\rm (i)]

\item When $\gamma >0$, we have
\begin{align}
\label{eq:2.12}
M\l( \phi_{\omega ,c} \r) =
\l\{ 
\begin{array}{ll}
\ds \frac{8}{\sqrt{\gamma}} \tan^{-1} \sqrt{ \frac{1+\beta}{1-\beta} } &\ds \text{\rm if}~-2\sqrt{\omega} <c<2\sqrt{\omega}, \\[10pt]
\ds \frac{4\pi}{\sqrt{\gamma}} &\ds\text{\rm if}~c=2\sqrt{\omega},
\end{array}
\r.
\end{align}
where $\beta$ is defined by
\begin{align}
\label{eq:2.13}
\beta =\beta(\omega ,c):= \frac{c}{ \sqrt{c^2+\gamma (4\omega -c^2)} }.
\end{align}

\item When $\gamma =0$, we have
\begin{align}
\label{eq:2.14}
M\l( \phi_{\omega ,c} \r) =\frac{4\sqrt{4\omega -c^2}}{-c}\quad \text{\rm if}~-2\sqrt{\omega} <c<0.
\end{align}

\item When $\gamma <0$, we have
\begin{align}
\label{eq:2.15}
M\l( \phi_{\omega ,c} \r) =\frac{4}{\sqrt{-\gamma}} \log \l( \alpha +\sqrt{\alpha^2 -1}\r) 
\quad\text{\rm if}~-2\sqrt{\omega} <c<-2s_{\ast}\sqrt{\omega},
\end{align}
where $\alpha$ is defined by
\begin{align}
\label{eq:2.16}
\alpha =\alpha (\omega ,c):= \frac{-c}{ \sqrt{c^2+\gamma (4\omega -c^2)} }.
\end{align}
\end{enumerate}
Moreover, if $\gamma >0$, the function
\begin{align*}
(-1, 1] \ni s \mapsto M\l(\phi_{1,2s} \r) \in \l( 0,  \frac{4\pi}{ \sqrt{\gamma} }\r]
\end{align*}
is continuous, strictly increasing and surjective. Similarly, if $\gamma\leq 0$ the function
\begin{align*}
\l( -1, -s_{\ast} \r) \ni s \mapsto M\l(\phi_{1,2s}\r) \in ( 0,  \infty )
\end{align*}
has the same property.
\end{proposition}
\begin{proof}
Let $(\omega ,c)$ and $\gamma$ satisfy \eqref{eq:2.9}. When $\omega >c^2/4$, from the explicit formulae of the solitons, we have
\begin{align}
\label{eq:2.17} 
M\l( \phi_{\omega ,c} \r) =M\l(\Phi_{\omega ,c} \r)&= \int_{-\infty}^{\infty} 
\frac{  2(4\omega -c^2)dx }{\sqrt{c^2+\gamma (4\omega -c^2)} \cosh ( \sqrt{4\omega- c^2}x)-c } \\
&= \frac{ 2\sqrt{4\omega -c^2} }{ \sqrt{c^2+\gamma (4\omega -c^2)} } \int_{-\infty}^{\infty} \frac{dy}{ \cosh y +\alpha },
\notag
\end{align} 
where $\alpha$ is defined by (\ref{eq:2.16}). 

Case 1-1: $\gamma >0$ and $-2\sqrt{\omega} <c<2\sqrt{\omega}$. In this case we note that $|\alpha| <1$ and
\begin{align}
\label{eq:2.18}
1-\alpha^2 &= 1-\frac{ c^2 }{ c^2+\gamma (4\omega -c^2) } =\frac{ \gamma (4\omega -c^2) }{ c^2+\gamma (4\omega -c^2) }.
\end{align}
Applying Lemma~\ref{lem:2.1} to (\ref{eq:2.17}), we have 
\begin{align}
\label{eq:2.19}
M\l( \phi_{\omega ,c}\r) 
&= \frac{ 2\sqrt{4\omega -c^2} }{ \sqrt{c^2+\gamma (4\omega -c^2)} } \cdot \frac{4}{\sqrt{1-\alpha^2}} \tan^{-1} 
\sqrt{ \frac{1-\alpha}{1+\alpha} } \\
&= \frac{8}{ \sqrt{\gamma} }\tan^{-1} \sqrt{ \frac{1+\beta}{1-\beta} },
\notag
\end{align}
where $\beta:=-\alpha$. We note that the function $(\omega ,c)\mapsto \beta(\omega ,c)$ is constant on the scaling curve \eqref{eq:1.11}. For $s\in (-1,1]$ we have
\begin{align*}
\beta (s):=\beta (\omega ,2s\sqrt{\omega} ) 
&=\frac{s}{ \sqrt{s^2+\gamma (1-s^2)} }= \frac{ \sgn~s}{ \sqrt{1+\gamma \l( \frac{1}{s^2} -1\r)} }.
\end{align*}
This shows that the function
\begin{align}
\label{eq:2.20}
(-1,1] \ni s \mapsto \beta(s) \in (-1,1]
\end{align}
is continuous, strictly increasing and surjective. 
Therefore, by (\ref{eq:2.19}) we obtain that the function
\begin{align*}
(-1, 1) \ni s \mapsto M\l( \phi_{1,2s} \r) \in \l( 0,  \frac{4\pi}{ \sqrt{\gamma} }\r)
\end{align*}
has the same property. We also note that 
\begin{align}
\label{eq:2.21}
\lim_{s\to 1-0} M\l( \phi_{1,2s} \r) = \frac{4\pi}{ \sqrt{\gamma} }. 
\end{align}

Case 1-2: $\gamma >0$ and $c=2\sqrt{\omega}$. From the explicit formulae of algebraic solitons, we have
\begin{align}
\label{eq:2.22}
M\l(\phi_{c^2/4 ,c}\r) =M\l(\Phi_{c^2/4 ,c}\r) &= \int_{-\infty}^{\infty} \frac{4c}{c^2x^2 +\gamma} dx =\frac{4\pi}{ \sqrt{\gamma} }.
\end{align}
From (\ref{eq:2.21}) and (\ref{eq:2.22}), we obtain that
\begin{align}
\label{eq:2.23}
\lim_{s\to 1-0} M\l( \phi_{1,2s} \r) = M\l( \phi_{1,2}\r),
\end{align}
which completes the proof of the case $\gamma >0$.

Case 2: $\gamma =0$ and $-2\sqrt{\omega} <c<0$. In this case we note that $\alpha =1$. From (\ref{eq:2.17}) and Lemma~\ref{lem:2.1}, we have
\begin{align}
\label{eq:2.24}
M\l( \phi_{\omega ,c}\r) &= \frac{ 2\sqrt{4\omega -c^2} }{ -c } \int_{-\infty}^{\infty} \frac{dy}{ \cosh y +1 } 
=\frac{4\sqrt{4\omega -c^2}}{-c}.
\end{align}
For $s\in (-1,0)$, we have
\begin{align*}
M\l( \phi_{1,2s} \r) 
&=\frac{4\sqrt{1-s^2}}{-s}=4\sqrt{\frac{1}{s^2}-1},
\end{align*}
which yields that the function
\begin{align}
\label{eq:2.25}
(-1, 0) \ni s \mapsto M\l( \phi_{1,2s} \r)  \in (0,\infty )
\end{align}
is continuous, strictly increasing and surjective. 

Case 3: $\gamma <0$ and $-2\sqrt{\omega} <c<-2s_{\ast}\sqrt{\omega}$. In this case we note that $\alpha >1$. From Lemma~\ref{lem:2.1}, (\ref{eq:2.17}) and (\ref{eq:2.18}), we have
\begin{align}
\label{eq:2.26}
M\l( \phi_{\omega ,c}\r) 
&= \frac{ 2\sqrt{4\omega -c^2} }{ \sqrt{c^2+\gamma (4\omega -c^2)} } \int_{-\infty}^{\infty} \frac{dy}{ \cosh y +\alpha } \\
&= \frac{ 2\sqrt{4\omega -c^2} }{ \sqrt{c^2+\gamma (4\omega -c^2)} } \cdot \frac{2}{\sqrt{\alpha^2 -1}} \log \l( \alpha +\sqrt{\alpha^2 -1}\r)
\notag\\
&= \frac{4}{\sqrt{-\gamma}} \log \l( \alpha +\sqrt{\alpha^2 -1}\r).
\notag
\end{align}
We note that
\begin{align}
\label{eq:2.27}
\alpha (s) :=\alpha (\omega ,2s\sqrt{\omega} ) 
&= \frac{-s}{ \sqrt{(1-\gamma )s^2 +\gamma} }= \frac{1}{ \sqrt{1-\gamma +\gamma s^{-2} } }.
\end{align}
This yields that the function
\begin{align*}
\l( -1,  -s_{\ast} \r) \ni s \mapsto \alpha (s) \in (1, \infty )
\end{align*}
is continuous, strictly increasing and surjective. From the formula (\ref{eq:2.26}), we deduce that the function
\begin{align}
\label{eq:2.28}
\l( -1,  -s_{\ast}\r) \ni 
s \mapsto M\l( \phi_{1,2s} \r) \in (0,\infty )
\end{align}
has the same property. This completes the proof.
\end{proof}
\subsection{Momentum of the solitons}%
\label{sec:2.3}
In this subsection we calculate the momentum of the solitons. From the formula (\ref{eq:2.4}) of the solitons, we have
\begin{align}
\label{eq:2.29}
P(\phi _{\omega ,c})&=\re \int_{\R} i\phi_{\omega ,c}' \overline{\phi_{\omega ,c}} dx
=-\frac{c}{2}M(\Phi_{\omega ,c}) +\frac{1}{4} \| \Phi_{\omega ,c}\|_{L^4}^4.
\end{align}
To calculate the $L^4$-norm, we prepare the following elementary integration formulae.
\begin{lemma}
\label{lem:2.3}
Let $-1<\alpha $. Then we have
\begin{align}
\label{eq:2.30}
\thickmuskip=0mu 
\medmuskip=0mu 
\thinmuskip=0mu 
\int_{-\infty}^{\infty} \frac{dy}{(\cosh y+\alpha )^2} =\l\{
\begin{array}{ll}
\thickmuskip=0mu 
\medmuskip=0mu 
\thinmuskip=0mu 
\ds \frac{2}{1-\alpha^2} -\frac{4\alpha}{ (1-\alpha^2 )^{3/2} }\tan^{-1} \sqrt{ \frac{1-\alpha}{1+\alpha} } 
&\text{\rm if}~~|\alpha |<1,\\[10pt]
\ds \quad \frac{2}{3} &\text{\rm if}~~\alpha =1, \\[10pt]
\thickmuskip=0mu 
\medmuskip=0mu 
\thinmuskip=0mu 
\ds -\frac{2}{\alpha^2 -1}+\frac{2\alpha}{ (\alpha^2 -1)^{3/2} } \log \l( \alpha +\sqrt{\alpha^2 -1}\r) 
&\text{\rm if}~\alpha >1.
\end{array}
\r.
\end{align}
\end{lemma}
\begin{proof}
Change variables $t=e^y$ and apply the formula 3.252, 4 in \cite{GR07}. 
\end{proof}
By using Lemma \ref{lem:2.3}, we have the following proposition.
\begin{proposition}
\label{prop:2.4}
The momentum of the solitons is represented as follows\textup{;} if $\gamma >0$ and $-2\sqrt{\omega} <c\leq 2\sqrt{\omega}$, or if $\gamma <0$ and $-2\sqrt{\omega}<c<-2s_{\ast}\sqrt{\omega}$, we have
\begin{align}
\label{eq:2.31}
P(\phi_{\omega ,c}) = \frac{c}{2} \l( -1+ \frac{1}{\gamma} \r) M(\phi_{\omega ,c}) +\frac{2}{\gamma}\sqrt{4\omega -c^2}.
\end{align}
If $\gamma =0$ and $-2\sqrt{\omega}<c<0$, we have
\begin{align}
\label{eq:2.32}
P(\phi_{\omega ,c}) = - \frac{2\omega +c^2}{3c} M(\phi_{\omega ,c}) .
\end{align}
\end{proposition}
\begin{remark}
\label{rem:2.5}
The momentum is represented by the same formula in the cases $\gamma >0$ and $\gamma<0$ although each mass is represented by the different functions. 
\end{remark}
\begin{proof}
We only consider the case $\omega >c^2/4$. The case $c=2\sqrt{\omega}$ is calculated more easily. We note that $\Phi_{\omega ,c}^2(x)$ is rewritten as
\begin{align*}
\Phi_{\omega ,c}^2(x) 
&= \frac{ 2(4\omega -c^2) }{ \sqrt{c^2+\gamma (4\omega -c^2)} }\cdot\frac{1}{ \cosh ( \sqrt{4\omega- c^2}x) +\alpha },
\end{align*}
where $\alpha$ is defined by (\ref{eq:2.16}). Then we have
\begin{align}
\label{eq:2.33}
\| \Phi_{\omega ,c}\|_{L^4}^4 
&= \frac{ 4(4\omega -c^2)^{3/2} }{ c^2+\gamma (4\omega -c^2) } \int_{-\infty}^{\infty} \frac{dy}{ (\cosh y +\alpha )^2 }.
\end{align}

Case 1: $\gamma >0$ and $-2\sqrt{\omega} <c<2\sqrt{\omega}$. In this case we note that $|\alpha| <1$. From Lemma \ref{lem:2.3}, (\ref{eq:2.18}) and (\ref{eq:2.12}), we obtain that
\begin{align}
\label{eq:2.34}
\| \Phi_{\omega ,c}\|_{L^4}^4 
&= \frac{ 4(4\omega -c^2)^{3/2} }{ c^2+\gamma (4\omega -c^2) }\cdot
\l[  \frac{2}{1-\alpha^2} -\frac{4\alpha}{ (1-\alpha^2 )^{3/2} }\tan^{-1} \sqrt{ \frac{1-\alpha}{1+\alpha} }\r] 
\\
 &=\frac{8}{\gamma} \sqrt{4\omega -c^2} +\frac{16c}{\gamma^{3/2}} \tan^{-1} \sqrt{ \frac{1+\beta}{1-\beta} } 
\notag\\
&=\frac{8}{\gamma} \sqrt{4\omega -c^2} +\frac{2c}{\gamma} M(\Phi_{\omega ,c}).
\notag
\end{align}
From (\ref{eq:2.29}) and (\ref{eq:2.34}), we have
\begin{align*}
P(\phi_{\omega ,c}) &=-\frac{c}{2}M(\Phi_{\omega ,c}) +\frac{1}{4} \| \Phi_{\omega ,c}\|_{L^4}^4 \\
&=\frac{c}{2} \l( -1+ \frac{1}{\gamma} \r) M(\Phi_{\omega ,c}) +
\frac{2}{\gamma}\sqrt{4\omega -c^2}.
\end{align*}

Case 2: $\gamma <0$ and $-2\sqrt{\omega} <c<0$. In this case we note that $\alpha =1$. By Lemma \ref{lem:2.3} and (\ref{eq:2.14}), we obtain that
\begin{align}
\label{eq:2.35}
\| \Phi_{\omega ,c}\|_{L^4}^4 
&= \frac{ 4(4\omega -c^2)^{3/2} }{ c^2} \int_{-\infty}^{\infty} \frac{dy}{ (\cosh y +1 )^2 } 
\\
&= -\frac{2(4\omega -c^2)}{3c} M(\Phi_{\omega ,c}).
\notag
\end{align}
From (\ref{eq:2.29}) and (\ref{eq:2.35}), we have
\begin{align*}
P(\phi_{\omega ,c}) &=-\frac{c}{2} M(\Phi_{\omega ,c}) +\frac{1}{4} \| \Phi_{\omega ,c}\|_{L^4}^4\\
&=- \frac{2\omega +c^2}{3c} M(\Phi_{\omega ,c}).
\end{align*}

Case 3: $\gamma >0$ and $-2\sqrt{\omega} <c<-2s_{\ast}\sqrt{\omega}$. In this case we note that $\alpha >1$. By Lemma \ref{lem:2.3}, (\ref{eq:2.18}) and (\ref{eq:2.15}), we obtain that
\begin{align}
\label{eq:2.36}
\| \Phi_{\omega ,c}\|_{L^4}^4 
&
= \frac{ 4(4\omega -c^2)^{3/2} }{ c^2+\gamma (4\omega -c^2) }\cdot
\l[\! 
-\frac{2}{\alpha^2 -1} +\frac{2\alpha}{ (\alpha^2 -1)^{3/2} } \log\l( \alpha +\sqrt{\alpha^2 -1} \r) 
\!\r] 
\\
&=\frac{8}{\gamma} \sqrt{4\omega -c^2} -\frac{8c}{(-\gamma )^{3/2}} \log\l( \alpha +\sqrt{\alpha^2 -1} \r)   
\notag\\
&=\frac{8}{\gamma} \sqrt{4\omega -c^2} -\frac{2c}{-\gamma} M(\Phi_{\omega ,c}).
\notag
\end{align}
This is exactly the same as the formula (\ref{eq:2.34}). Hence the momentum has the same formula as the Case 1. This completes the proof.
\end{proof}
By the Pohozaev identity, the energy of the soliton is represented by the momentum.
\begin{proposition}
\label{prop:2.6}
Let $(\omega ,c)$ and $\gamma$ satisfy \eqref{eq:2.9}. Then we have
\begin{align}
\label{eq:2.37}
E(\phi_{\omega ,c}) =-\frac{c}{4} P(\phi_{\omega ,c}).
\end{align}
\end{proposition}
\begin{proof}
For completeness we give a proof. Let $\phi^{\lambda}(x) =\lambda^{1/2}\phi(\lambda x)$ for $\lambda >0$. 
Then we have
\begin{align}
\label{eq:2.38}
S_{\omega ,c} (\phi_{\omega ,c}^{\lambda}) &= E(\phi_{\omega ,c}^{\lambda}) 
+\frac{\omega}{2}M(\phi_{\omega ,c}^{\lambda}) 
+\frac{c}{2}P (\phi_{\omega ,c}^{\lambda}) 
\\
&=\lambda^2 E(\phi_{\omega ,c}) +\frac{\omega}{2}M(\phi_{\omega ,c})+\frac{c\lambda}{2}P(\phi_{\omega ,c}).
\notag
\end{align}
Since $S_{\omega ,c}' (\phi_{\omega ,c}) =0$, we deduce that
\begin{align*}
0= \l. \frac{d}{d\lambda}  S_{\omega ,c} (\phi_{\omega ,c}^{\lambda}) \r|_{\lambda =1} 
=2E(\phi_{\omega ,c}) + \frac{c}{2}P(\phi_{\omega ,c}).
\end{align*}
Hence the result follows.
\end{proof}
\subsection{Positivity of the momentum}%
\label{sec:2.4}
The effect of the momentum plays an essential role in the potential well theory. In this subsection we study the sign of the momentum of the soliton. For $(\omega ,c)$ satisfying \eqref{eq:2.9}, we rewrite
$(\omega ,c) =(\omega ,2s\sqrt{\omega})$,
where the parameter $s$ satisfies 
\begin{align}
\label{eq:2.39}
\begin{array}{ll}
\ds\text{if}~b>-3/16,& \ds -1 <s\leq 1,\\[7pt]
\ds\text{if}~b\leq -3/16,& \ds -1 <s<-s_{\ast}. 
\end{array}
\end{align}
Since $P(\phi_{\omega ,2s\sqrt{\omega}}) =\sqrt{\omega}P(\phi_{1,2s})$, it is enough to check the sign of $P(\phi_{1,2s})$. 
\begin{figure}[t]
\begin{minipage}{0.49\linewidth} 
 \includegraphics[width=\linewidth]{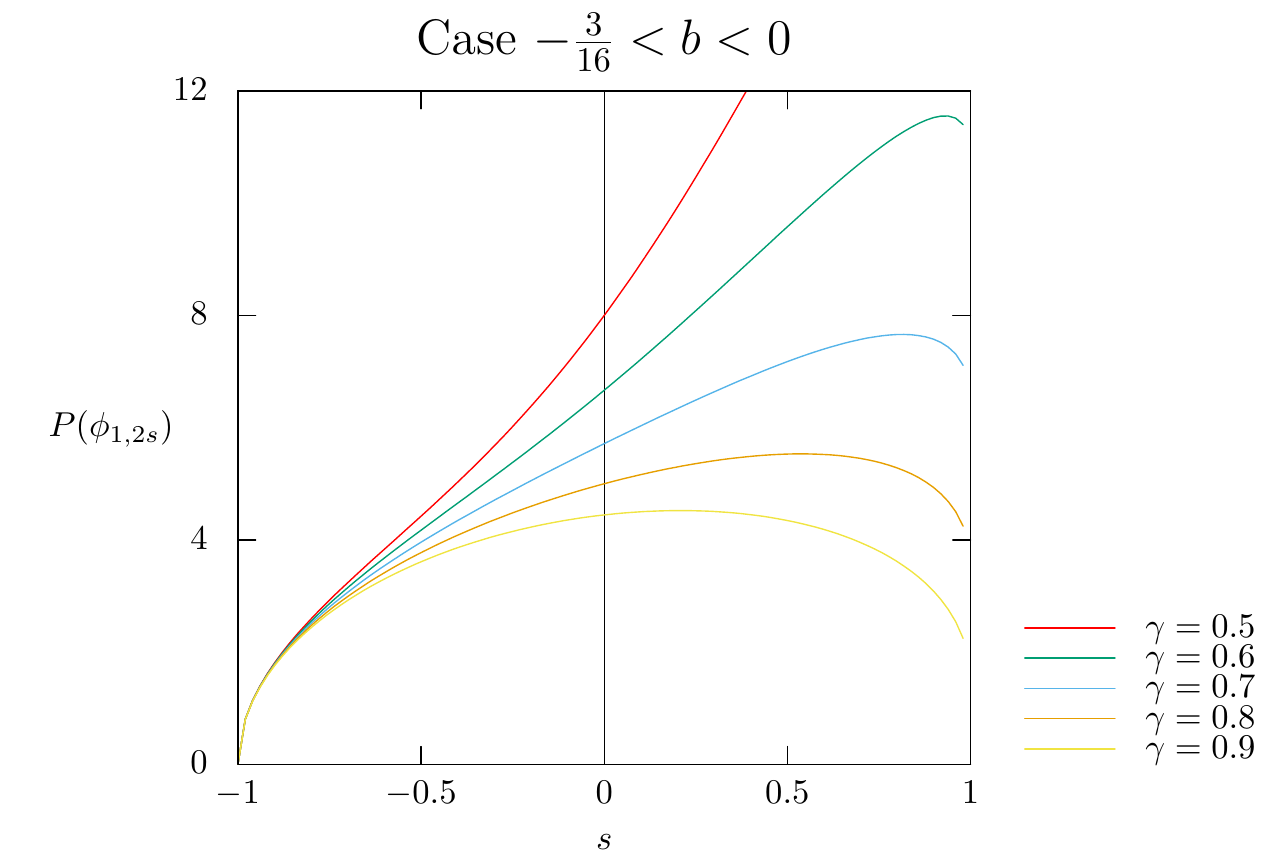}
\end{minipage}
\begin{minipage}{0.49\linewidth}
 \includegraphics[width=\linewidth]{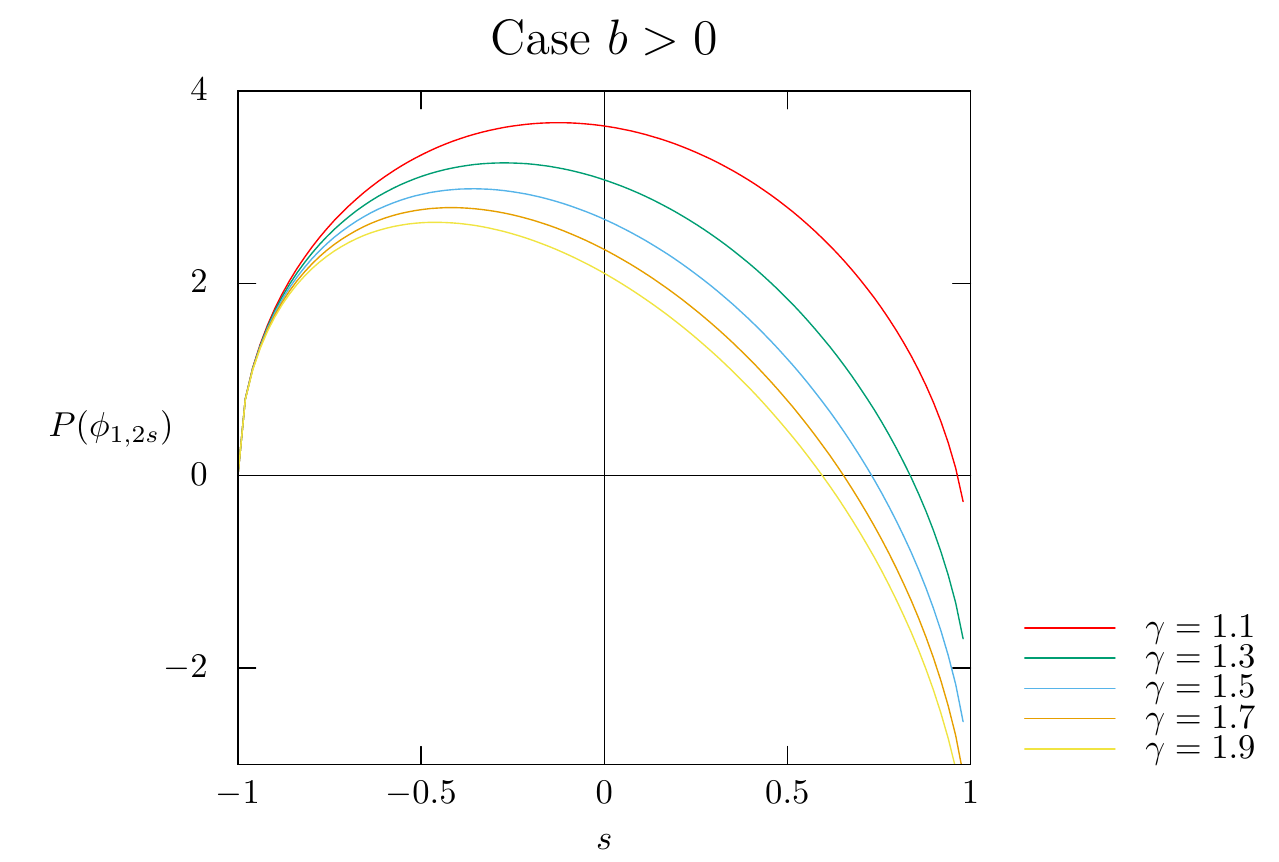}
\end{minipage}
\caption{The function $s\mapsto P(\phi_{1,2s})$ for several values of $b>-3/16$.}
\label{fig:3}
\end{figure}
\begin{proposition}
\label{prop:2.7}
Let $s$ satisfy \eqref{eq:2.39}. Then the following properties hold\textup{:}
\begin{enumerate}[\rm (i)]
\setlength{\itemsep}{2pt}

\item If $b<0$, $P(\phi_{1,2s})>0$ for any $s$.

\item If $b =0$, $P(\phi_{1,2s}) >0$ for $s\in (-1,1)$ and $P(\phi_{1,2})=0$.

\item If $b>0$, there exists a unique $\thickmuskip=1mu s^{*}=s^*(b)\in (0,1)$ such that $\thickmuskip=0mu\thinmuskip=0mu  P(\phi_{1,2s^*})=0$. Moreover, we have $P(\phi_{1,2s})>0$ for $s\in (-1, s^*)$ and $P(\phi_{1,2s})<0$ for $s\in(s^{*},1]$. 
\end{enumerate}
\end{proposition}
\begin{remark}
\label{rem:2.8}
The existence of $s^*$ in (iii) was first proved in \cite{O14}. As in Figure \ref{fig:3}, the zero point of the function $s\mapsto P(\phi_{1,2s})$ moves to the right and converges to $1$ as $b\downarrow 0$. This remark is rigorously proved below.
\end{remark}
\begin{proof}
From the formula \eqref{eq:2.29}, $P(\phi_{1,2s})$ is always positive when $s\leq 0$. Hence we need only consider the case $s>0$.

(i) It is enough to consider the case $-3/16 <b<0$. First we note that the formula (\ref{eq:2.31}) is rewritten as
\begin{align}
\label{eq:2.40}
P(\phi_{1,2s}) =s\l( -1+\frac{1}{\gamma} \r) M(\phi_{1,2s}) +\frac{4}{\gamma}\sqrt{1-s^2}.
\end{align}
Since $-1+\frac{1}{\gamma}>0$, it follows from (\ref{eq:2.40}) that $P(\phi_{1,2s}) >0$ for $s\in(0,1]$. 

(ii) This is obvious from the formula (\ref{eq:2.40}).

(iii) When $b>0$, we note that
\begin{align*}
&P(\phi_{1,0}) =\frac{4}{\gamma}>0, \\
&P(\phi_{1,2}) = \l( -1+\frac{1}{\gamma} \r) M(\phi_{1,2})=
-\frac{4\pi\l(\gamma -1\r)}{\gamma^{3/2}} <0,
\end{align*}
and the function $[0,1]\ni s\mapsto P(\phi_{1,2s})$ is continuous and strictly decreasing. Therefore there exists $s^{*}\in (0,1)$ such that $P(\phi_{1,2s^*})=0$, $P(\phi_{1,2s})>0$ for $s\in (0, s^*)$ and $P(\phi_{1,2s})<0$ for $s\in(s^{*},1]$. This completes the proof.
\end{proof}
We define a function $(\omega ,c)\mapsto d(\omega ,c)$ by
\begin{align}
\label{eq:2.41}
d(\omega ,c) := S_{\omega ,c} (\phi_{\omega ,c}). 
\end{align}
We note that $d(\omega ,2s\sqrt{\omega})=\omega d(1,2s)$. From Proposition \ref{prop:2.7} we obtain 
the following key lemma on the proof of Theorem \ref{thm:1.7}. 
\begin{lemma}
\label{lem:2.9}
Let $s$ satisfy \eqref{eq:2.39}. Then the following properties hold\textup{:}
\begin{enumerate}[\rm (i)]

\item If $b>0$, the function $(-1,1]\ni s \mapsto d(1,2s) $ is strictly increasing on $(-1,s^*)$ and strictly decreasing on $(s^* ,1]$. 

\item If $-3/16<b\leq 0$, the function $(-1,1]\ni s \mapsto d(1,2s) $ is strictly increasing.

\item If $b\leq -3/16$, the function $(-1,-s_*)\ni s \mapsto d(1,2s) $ is strictly increasing.

\end{enumerate}
\end{lemma}
\begin{proof}
From the definition we have
\begin{align*}
d(1,2s) =S_{1,2s}(\phi_{1,2s})=E(\phi_{1,2s})+\frac{1}{2}M(\phi_{1,2s})+sP(\phi_{1,2s}).
\end{align*}
Since $S_{1,2s}'(\phi_{1,2s})=0$, we have
\begin{align*}
\frac{d}{ds} d(1,2s) =P(\phi_{1,2s}).
\end{align*}
Hence the result follows from Proposition \ref{prop:2.7}.
\end{proof}


\section{Gauge transformation and local well-posedness in $H^1(\R)$}
\label{sec:3}
In this section we review the gauge transformation and the local well-posedness theory in the energy space.
First we recall the result of local well-posedness for (\ref{eq:1.1}) in the energy space. 
\begin{theorem}[\cite{Oz96}]
\label{thm:3.1}
For every $u_0 \in H^1 (\R)$, there exist $0<T_{\rm min}, T_{\rm max}\leq \infty$ and a unique, maximal solution $u\in C((-T_{\rm min},T_{\rm max}) , H^1 (\R) )\cap L^4((-T_{\rm min},T_{\rm max}), W^{1,\infty}(\R ))$ of \eqref{eq:1.1} with $u(0)=u_0$. Furthermore, the following properties hold\textup{:}
\begin{enumerate}[\rm (i)]
\setlength{\itemsep}{3pt}

\item If $T_{\rm max}<\infty$\,{\rm (}resp., if  $T_{\rm min}<\infty${\rm )}, then $\|\del_x u(t)\|_{L^2}\!\to\!\infty$ as $t\uparrow T_{\rm max}$\,
{\rm (}resp., as $t\downarrow -T_{\rm min}${\rm )}.

\item There is conservation of energy, mass and momentum; i.e., $E(u(t))=E(u_0)$, $M(u(t))=M(u_0 )$ and $P(u(t))=P(u_0)$ for all $t \in (-T_{\min}, T_{\max})$.

\item Continuous dependence is satisfied in the following sense$;$ if $u_{0n}\to u_0$ in $H^1(\R )$ and if $I\subset (-T_{\rm min} (u_0 ) ,T_{\rm max} (u_0) )$ is a closed interval, then the maximal solution $u_n$ of \eqref{eq:1.1} with $u_n(0)=u_{0n}$ is defined on $I$ for $n$ large enough and satisfies $u_n \to u$ in $C(I, H^1 (\R))$.
\end{enumerate}
\end{theorem}
In \cite{Oz96} the proof of Theorem \ref{thm:3.1} is done by transforming the equation (\ref{eq:1.1}) into a new system of equations as follows; see also \cite{H93, HO92, HO94a}. For the solution $u$ of (\ref{eq:1.1}), we set 
\begin{align*}
\varphi (t,x) &= \exp\l( \frac{i}{2}\int_{-\infty}^{x}|u(t,y)|^2 dy\r) u(t,x) ,\\
\psi (t,x) &= \exp\l( \frac{i}{2}\int_{-\infty}^{x}|u(t,y)|^2 dy\r) \del_x u(t,x),
\end{align*}
then new functions $\varphi$ and $\psi$ formally satisfy
\begin{align}
\label{eq:3.1}
\l\{
\begin{array}{l}
i\del_t \varphi +\del^2_x \varphi =i\varphi^2\overline{\psi} +f(\varphi ),\\[3pt]
i\del_t \psi +\del^2_x \psi =-i\psi^2\overline{\varphi} 
+\del_{\varphi}f(\varphi )\psi +\del_{\overline{\varphi}}f(\varphi )\overline{\psi},
\end{array}
\r.
\end{align}
where $f(\varphi )=-b|\varphi|^4\varphi$. Since the system (\ref{eq:3.1}) has no loss of derivatives unlike the original equation (\ref{eq:1.1}), one can solve the Cauchy problem by the fixed point argument. However, in order to construct the solution of (\ref{eq:1.1}) through the system, we need to solve the equation (\ref{eq:3.1}) under the constraint condition
\begin{align*}
\psi =\del_x\varphi -\frac{i}{2}|\varphi|^2\varphi.
\end{align*}
This requires more or less complex calculation; see \cite{HO94a} for the details. We refer to \cite{HO16} for a more direct approach without using a system of equations.

We note that the gauge transformation plays a key role when one transforms the equation (\ref{eq:1.1}) into a system of equations (\ref{eq:3.1}). Here we consider more general gauge transformations as seen in \cite{Wu13}. For $a\in\R$ we define $\cG_a : H^1(\R)\to H^1(\R)$ by
\begin{align}
\label{eq:3.2}
\cG_a (u)(t,x) =\exp\l( ia\int_{-\infty}^{x} |u(t,y)|^2dy \r) u(t,x).
\end{align}
By a direct computation we have the following result.
\begin{proposition}
\label{prop:3.2}
Let $a\in\R$, and let $u\in C((-T_{\rm min},T_{\rm max}) , H^1 (\R) )$ be a maximal solution of \eqref{eq:1.1}. 
Then $v=\cG_a (u) \in C((-T_{\rm min},T_{\rm max}) , H^1 (\R) )$, and satisfies the following equation
\begin{align}
\label{eq:3.3}
i\del_t v+\del^2_x v+(-2a+1)i|v|^2\del_x v-2aiv^2\del_x\overline{v}+\l( a^2+\frac{a}{2}+b\r) |v|^4v=0.
\end{align}
Moreover, the equation \eqref{eq:3.3} has the following conserved quantities\textup{:}
\begin{align*}
E_a (v) &=\frac{1}{2} \| \del_x v\|_{L^2}^2+\l( a-\frac{1}{4}\r) 
\rbra[i|v|^2\del_x v, v]
+\l( \frac{a^2}{2}-\frac{a}{4}-\frac{b}{6}\r) \| v\|_{L^6}^6 ,
\\
M_a (v)&= \| v \|_{L^2}^2,
\\
P_a (v)&= \rbra[i\del_x v,v] +a\| v\|_{L^4}^4.
\end{align*}
\end{proposition}
\begin{remark}
We note that the functions $u$ and $\cG_a(u)$ are defined on the same maximal interval. 
The well-posedness in $H^1(\R)$ for the equations \eqref{eq:1.1} and \eqref{eq:3.3} is equivalent since $u\mapsto\scG_{a}(u)$ is locally Lipschitz continuous on $H^1(\R)$.
\end{remark}
It is important to choose the suitable parameter $a\in\R$ depending on the situation. If we set $a=1/2$, 
the term $i|v|^2\del_x v$ is removed in (\ref{eq:3.3}) and it is useful when one treats the Fourier restriction norm (see \cite{T99, CKSTT01, CKSTT02}). 

When $a=1/4$ the interaction term with derivative in the energy is canceled out, which is useful to derive a mass condition by using sharp Gagliardo--Nirenberg inequalities (see \cite{HO92, Wu13, Wu15}). 
In this paper we apply the gauge transformation in the case $a=1/4$ for giving variational characterization of the solitons including the case $b<0$. By Proposition \ref{prop:3.2}, $v=\cG_{1/4}(u)$ satisfies the equation
\begin{align}
\label{eq:3.4}
i\del_t v+\del_x^2 v+\frac{i}{2}|v|^2\del_x v-\frac{i}{2}v^2\del_x\overline{v}+\frac{3}{16}\gamma |v|^4v=0,
\quad \gamma =1+\frac{16}{3}b,
\end{align}
which is nothing but the equation \eqref{ME}. The conserved quantities of \eqref{eq:3.4} are as follows:
\begin{align*}
\tag{Energy}
 \cE (v) &=E_{1/4}(v)=\frac{1}{2} \| \del_x v\|_{L^2}^2-\frac{\gamma}{32} \| v\|_{L^6}^6 ,
\\
\tag{Mass}
 \cM (v)&=M_{1/4}(v)=\| v \|_{L^2}^2,
\\
\tag{Momentum}
\cP (v)&=P_{1/4}(v)=\rbra[i\del_x v,v] +\frac{1}{4} \| v\|_{L^4}^4.
\end{align*}
We note that the energy functional $\cE(v)$ is nonnegative if $b \leq -3/16$. Hence we have the following result.
\begin{proposition}
\label{prop:3.4}
Let $b\leq -3/16$. For every $u_0 \in H^1(\R)$, the maximal $H^1(\R)$-solution $u$ of \eqref{eq:1.1} given by Theorem \ref{thm:3.1} is global and 
\begin{align*}
\sup_{t\in\R} \| u(t)\|_{H^1}\leq C(\| u_0\|_{H^1})<\infty .
\end{align*}
\end{proposition}
When $b>-3/16$, by applying the sharp Gagliardo--Nirenberg inequality
\begin{align}
\label{eq:3.5}
\| f\|_{L^6}^6 \leq \frac{4}{\pi^2}\| f\|_{L^2}^4\| \partial_{x}f\|_{L^2}^2\quad 
( \Leftrightarrow\eqref{GN1} ),
\end{align}
we deduce that if the initial data $u_0\in H^1(\R)$ satisfying $\| u_0\|_{L^2}^2 <\frac{2\pi}{\sqrt{\gamma}}$, then the corresponding solution is global and bounded. A similar approach was originally taken in \cite{HO92, HO94a, Oz96}.

Finally, we discuss the solitons of \eqref{eq:3.4}. Let $(\omega ,c )$ satisfy (\ref{eq:2.9}). The equation \eqref{eq:3.4} has a two-parameter family of solitons
\begin{align}
\label{eq:3.6}
v_{\omega ,c}(t ,x) =\cG_{1/4}(u_{\omega ,c})(t,x)=e^{i\omega t}\varphi_{\omega ,c}(x-ct),
\end{align}
where $\varphi_{\omega ,c}$ is defined by
\begin{align*}
\varphi_{\omega ,c} (x) =e^{i\frac{cx}{2}}\Phi_{\omega ,c}(x).
\end{align*}
We note that $\varphi_{\omega ,c}$ satisfies the equation
\begin{align}
\label{eq:3.7}
 -\varphi'' +\omega\varphi +ic\varphi' +\frac{c}{2}|\varphi |^2\varphi
-\frac{3}{16}\gamma |\varphi|^4 \varphi =0, \quad x\in \R.
\end{align}
We note that \eqref{eq:3.7} is rewritten as $\cS_{\omega ,c} ' (\varphi )=0$, where 
\begin{align*}
\cS_{\omega ,c} (\varphi ) =\cE (\varphi )+\frac{\omega}{2}\cM (\varphi )
+\frac{c}{2}\cP (\varphi ).
\end{align*}
For the action functionals $S_{\omega ,c}$ and $\cS_{\omega ,c}$, we have the following relation:
\begin{align*}
S_{\omega ,c}(u)=\cS_{\omega ,c} (\cG_{1/4} (u) ) \quad \text{for any}~u\in H^1(\R). 
\end{align*}
In particular, we have
\begin{align}
\label{eq:3.8}
d(\omega ,c) =S_{\omega ,c} (\phi_{\omega ,c}) =\cS_{\omega ,c} (\varphi_{\omega ,c}). 
\end{align}

\section{Variational characterization}
\label{sec:4}
In this section we give a variational characterization of the soliton $v_{\omega ,c}$ defined by (\ref{eq:3.6}). Here we assume that $\gamma$ and $(\omega ,c)$ satisfy
\begin{align} 
\label{eq:4.1}
\begin{array}{ll}
\ds\text{if}~\gamma >0\Leftrightarrow b>-3/16,& \ds -2\sqrt{\omega} <c\leq 2\sqrt{\omega} ,\\[7pt]
\ds\text{if}~\gamma = 0\Leftrightarrow b= -3/16,& \ds -2\sqrt{\omega} <c<0.
\end{array}
\end{align}
We prepare some notations. First we define function spaces by
\begin{align}
\label{eq:4.2}
\varphi \in X_{\omega,c} &\iff
\l\{
\begin{array}{ll}
\ds \varphi \in H^1(\R) & \ds\text{if}~ \omega >c^2/4,
\\[5pt]
\ds e^{-i\frac{cx}{2}}\varphi \in \dot{H}^1(\R) \cap L^{4}(\R) & \ds\text{if}~c=2\sqrt{\omega},
\end{array}
\r.\\[3pt]
\| \varphi\|_{ X_{c^2/4,c} }&:=\| e^{-i\frac{c}{2}\cdot}\varphi\|_{\dot{H}^1\cap L^4}.
\notag
\end{align}
Note that $H^1(\R) \subset X_{c^2/4,c}$. We consider the functional $\cK_{\omega ,c}(\varphi )=\l.\frac{d}{d\lambda}\cS_{\omega ,c}(\lambda u)\r|_{\lambda=1}$ which has the following explicit formula:
\begin{align}
\label{eq:4.3}
\cK_{\omega ,c}(\varphi ) := \| \del_x\varphi\|_{L^2}^2+\omega\|\varphi\|_{L^2}^2
+c\rbra[i\del_x\varphi, \varphi] +\frac{c}{2}\| \varphi\|_{L^4}^4 -\frac{3}{16}\gamma\| \varphi\|_{L^6}^6.
\end{align}
We consider the following minimization problem:
\begin{align*}
\mu (\omega ,c) :=\inf \l\{ \cS_{\omega ,c}(\varphi) :\varphi \in X_{\omega ,c}\setminus \{ 0\} 
, \cK_{\omega ,c}(\varphi )=0  \r\}.
\end{align*}
We define the sets $\scG_{\omega ,c}$ and $\scM_{\omega ,c}$ by
\begin{align*}
\scG_{\omega ,c}&:=\l\{ \varphi \in X_{\omega ,c}\setminus\{ 0\} : 
\cS_{\omega ,c}'(\varphi )=0\r\},\\
\scM_{\omega ,c}&:=\l\{ \varphi \in X_{\omega ,c}\setminus\{ 0\} :
\cS_{\omega ,c}(\varphi) =\mu (\omega ,c), \cK_{\omega ,c}(\varphi )=0  \r\}.
\end{align*}
$\scG_{\omega ,c}$ is the set of nontrivial critical points of $\cS_{\omega ,c}$, and $\scM_{\omega ,c}$ is the set of minimizers of $\cS_{\omega ,c}$ on the Nehari manifold. The main result in this section is the following result.
\begin{proposition}
\label{prop:4.1}
Let $\gamma$ and $(\omega ,c)$ satisfy \eqref{eq:4.1}. Then we have
\begin{align}
\label{eq:4.4}
\scG_{\omega ,c}=\scM_{\omega ,c} =\l\{ e^{i\theta}\varphi_{\omega ,c}(\cdot -y):\theta\in[0,2\pi), y\in\R \r\},
\end{align}
and $d(\omega ,c)=\mu (\omega ,c)$.
\end{proposition}
Our proof of Proposition \ref{prop:4.1} depends on concentration compactness arguments in \cite{CO06} 
(see \cite{FHI17} for the case $c=2\sqrt{\omega}$). For convenience of notations, we define
\begin{align*}
\cL_{\omega ,c} (\varphi )&:=\|\del_x\varphi\|_{L^2}^2+\omega\|\varphi\|_{L^2}^2
+c\rbra[i\del_x\varphi, \varphi] ,\\
\cI_{\omega ,c}(\varphi )&:=\cS_{\omega,c}(\varphi )-\frac{1}{4}\cK_{\omega ,c}(\varphi )
=\frac{1}{4}\cL_{\omega ,c}(\varphi )+\frac{\gamma}{64}\|\varphi\|_{L^6}^6.
\end{align*}
First we prove the following lemma.
\begin{lemma}
\label{lem:4.2}
Let $\gamma$ and $(\omega ,c)$ satisfy \eqref{eq:4.1}. Then the following properties hold\textup{:}
\begin{enumerate}[\rm  (i)]
\item If $\omega >c^2/4$, there exists $C_1=C_1(\omega ,c)$ such that
\begin{align*}
\cL_{\omega ,c}(\varphi ) \geq C_1\| \varphi\|_{H^1}^2~\text{for}~\varphi\in H^1(\R ).
\end{align*}
\item $\mu (\omega ,c) >0$.

\item If $\varphi \in X_{\omega ,c}$ satisfies $\cK_{\omega ,c}(\varphi )<0$, 
then $\mu (\omega ,c)<\cI_{\omega ,c}(\varphi)$.

\end{enumerate}
\end{lemma}
\begin{proof}
(i) See Lemma 7 (1) in \cite{CO06}.

(ii) Case 1: $\omega >c^2/4$. 
Let $\varphi\in H^1(\R)\setminus\{0\}$ satisfy $\cK_{\omega ,c}(\varphi )=0$. By (i), (\ref{eq:4.3}) and the Sobolev inequality, there exists $C_2>0$ such that
\begin{align*}
C_1\|\varphi\|_{H^1}^2\leq\cL_{\omega ,c}(\varphi )
&=-\frac{c}{2}\|\varphi\|_{L^4}^4+\frac{3}{16}\gamma\|\varphi\|_{L^6}^6\\
&\leq\frac{|c|}{2}\|\varphi\|_{L^2}\| \varphi\|_{L^6}^3+\frac{3}{16}\gamma\| \varphi\|_{L^6}^6\\
&\leq\frac{C_1}{2}\| \varphi\|_{H^1}^2+C_2\| \varphi\|_{H^1}^6.
\end{align*}
This yields that $\| \varphi\|_{H^1}^4\geq \frac{C_1}{2C_2}$. Hence we have
\begin{align*}
\mu (\omega ,c)&=\inf\l\{ \cI_{\omega ,c}(\varphi ):\varphi \in H^1(\R )\setminus\{0\} 
, \cK_{\omega ,c}(\varphi )=0 \r\}\\
&\geq\frac{1}{4}\inf\l\{ \cL_{\omega ,c}(\varphi ):\varphi \in H^1(\R )\setminus\{0\} 
, \cK_{\omega ,c}(\varphi )=0 \r\}\\
&\geq\frac{C_1}{4}\sqrt{\frac{C_1}{2C_2}} >0.
\end{align*}
Case 2: $c=2\sqrt{\omega}$. In this case we have
\begin{align}
\label{eq:4.5}
\cL_{\omega ,c}(\varphi ) =\l\| \del_x\varphi -\frac{i}{2}c\varphi \r\|_{L^2}^2
 +\l( \omega -\frac{c^2}{4} \r) \| \varphi \|_{L^2}^2=\l\| \del_x\l( e^{-i\frac{cx}{2}}\varphi\r) \r\|_{L^2}^2>0
\end{align}
for $\varphi \in X_{\omega ,c}\setminus\{ 0\}$. This yields that $\mu (\omega , c)\geq 0$. We prove $\mu(\omega ,c) >0$ by contradiction. Assume that $\mu (\omega ,c)=0$. Then one can take the minimizing sequence $\{ \varphi_n\} \subset X_{\omega ,c}\setminus\{ 0\}$ such that
\begin{align}
\label{eq:4.6}
\cS_{\omega ,c} (\varphi_n ) \underset{n\to\infty}{\longrightarrow} 0, ~\text{and}~\cK_{\omega ,c}(\varphi_n )=0~\text{for all}~n\in\N.
\end{align}
Since $\cS_{\omega ,c}$ is rewritten as 
\begin{align}
\label{eq:4.7}
\cS_{\omega ,c}(\varphi )=\frac{1}{4}\cK_{\omega ,c}(\varphi )+\frac{1}{4}\cL_{\omega ,c}(\varphi ) +\frac{\gamma}{64}\| \varphi\|_{L^6}^6,
\end{align}
from (\ref{eq:4.5}) and (\ref{eq:4.6}) we obtain that
\begin{align*}
\l\| \del_x\l( e^{-i\frac{cx}{2}}\varphi_n \r) \r\|_{L^2}, ~\|\varphi_n\|_{L^6} \longrightarrow 0
\end{align*}
as $n\to\infty$. By using an elementary interpolation inequality
\begin{align*}
\| f\|_{L^{\infty}}^4 \leq 4\| f\|_{L^6}^3\| \del_x f\|_{L^2},
\end{align*}
we have $\| \varphi_n\|_{L^{\infty}}\to 0$ as $n\to\infty$. Hence we have
\begin{align*}
0=\cK_{\omega ,c}(\varphi_n )&=\cL_{\omega ,c}(\varphi_n) 
+\frac{c}{2}\| \varphi_n\|_{L^4}^4 -\frac{3}{16}\gamma \| \varphi_n\|_{L^6}^6\\
&\geq\l( \frac{c}{2}-\frac{3}{16}\gamma\| \varphi_n\|_{L^{\infty}}^2\r) \|\varphi_n\|_{L^4}^4
>0
\end{align*}
for large $n\in\N$, which is a contradiction with \eqref{eq:4.6}.

(iii) Let $\varphi\in X_{\omega ,c}\setminus\{ 0\}$ satisfy $\cK_{\omega ,c}(\varphi) <0$. Then there exists a unique $\lambda _0 \in (0,1)$ such that $\cK_{\omega ,c}(\lambda_0\varphi ) =0$. From the definition of $\mu(\omega ,c)$, we have
\begin{align*}
\mu(\omega ,c) \leq \cI_{\omega ,c}(\lambda_0 \varphi )
=\frac{\lambda_0^2}{4}\cL_{\omega ,c}(\varphi )+\frac{\lambda_0^6\gamma}{64}\| \varphi\|_{L^6}^6
<\cI_{\omega ,c}(\varphi ).
\end{align*}
This completes the proof.
\end{proof}
By the standard ODE arguments (see e.g. \cite{C03, FHI17}), we have the following lemma.
\begin{lemma}
\label{lem:4.3}
Let $\gamma$ and $(\omega ,c)$ satisfy \eqref{eq:2.9}. Then we have
\begin{align*}
\scG_{\omega ,c} =\l\{ e^{i\theta}\varphi_{\omega ,c}(\cdot -y):\theta\in[0,2\pi), y\in\R \r\}.
\end{align*}
\end{lemma}
Next we prove the following result.
\begin{lemma}
\label{lem:4.4}
Let $\gamma$ and $(\omega ,c)$ satisfy \eqref{eq:4.1}. Assume that $\scM_{\omega ,c}\neq \emptyset$. Then we have $\scG_{\omega, c}=\scM_{\omega ,c}$. Moreover we have 
$d(\omega ,c)=\mu(\omega ,c)$.
\end{lemma}
\begin{proof}
First we prove $\scM_{\omega ,c}\subset \scG_{\omega ,c}$. Let $\varphi \in\scM_{\omega ,c}$. Since $\varphi$ is a minimizer on the Nehari manifold, there exists a Lagrange multiplier $\eta\in\R$ such that $\cS_{\omega ,c}'(\varphi )=\eta\cK_{\omega ,c}'(\varphi )$. Thus we have
\begin{align*}
0=\cK_{\omega ,c}(\varphi )=\tbra[\cS_{\omega ,c}'(\varphi ), \varphi] =\eta\tbra[\cK_{\omega ,c}'(\varphi ), \varphi ].
\end{align*}
By $\cK_{\omega ,c}(\varphi )=0$ and $\varphi\neq 0$, we have
\begin{align*}
\tbra[\cK_{\omega ,c}'(\varphi ),\varphi ]&=2\cL_{\omega ,c}(\varphi )+2c\|\varphi\|_{L^4}^4
-\frac{9}{8}\gamma\| \varphi\|_{L^6}^6\\
&=-2\cL_{\omega ,c}(\varphi )-\frac{3}{8}\gamma\|\varphi\|_{L^6}^6<0.
\end{align*}
This yields that $\eta=0$ and $\varphi\in\scG_{\omega ,c}$, which implies $\scM_{\omega ,c}\subset \scG_{\omega ,c}$. Conversely, let $\varphi\in\scG_{\omega ,c}$. By Lemma \ref{lem:4.3}, there exist $\theta_0\in[0,2\pi )$ and $y_0\in\R$ such that $\varphi =e^{i\theta_0}\varphi_{\omega ,c}(\cdot -y_0)$. Since $\scM_{\omega ,c}\neq\emptyset$, we can take some $\psi\in\scM_{\omega ,c}$. By $\scM_{\omega ,c}\subset \scG_{\omega ,c}$ and Lemma \ref{lem:4.3}, there exist $\theta_1\in[0,2\pi )$ and $y_1\in\R$ such that $\psi =e^{i\theta_1}\varphi_{\omega ,c}(\cdot -y_1)$. Thus we have
\begin{align*}
\cS_{\omega ,c}(\varphi ) =\cS_{\omega ,c}(\varphi_{\omega ,c}) 
=\cS_{\omega ,c}(\psi ) = \mu (\omega ,c).
\end{align*}
This yields that $\varphi\in\scM_{\omega ,c}$ since $\cK_{\omega ,c}(\varphi )=\tbra[\cS_{\omega ,c}'(\varphi ), \varphi]=0$. This completes the proof. 
\end{proof}
To complete the proof of Proposition \ref{prop:4.1}, we need to prove that $\scM_{\omega ,c} \neq\emptyset$. 
To this end, we prepare two useful lemmas on concentration compactness.
\begin{lemma}[\cite{Lie83, BFV14}]
\label{lem:4.5}
Let $p\geq 2$. Let $\{f_n\}$ be a bounded sequence in $\dot{H}^1(\R)\cap L^{p}(\R)$. Assume that there exists $q\in(p,\infty)$ such that $\limsup_{n \to \infty} \norm[f_n]_{L^q}>0$.
Then, there exist $\{y_n\}\subset\R$ and $f \in \dot{H}^1(\R)\cap L^{p}(\R)\setminus \{0\}$ such that $\{f_n(\cdot-y_n)\}$ has a subsequence that converges to $f$ weakly in $\dot{H}^1(\R)\cap L^{p}(\R)$. 
\end{lemma}
\begin{lemma}[{\cite{BL83}}]
\label{lem:4.6}
Let $1\leq p < \infty$. Let $\{f_n\}$ be a bounded sequence in $L^p(\R)$ and $f_n \to f$ a.e. in $\R$ as $n\to \infty$. Then we have 
\begin{align*}
\| f_n\|_{L^p}^p - \| f_n-f\|_{L^p}^p - \| f \|_{L^p}^p \to 0
\end{align*}
as $n \to \infty$. 
\end{lemma}
The assertion $\scM_{\omega ,c}\neq\emptyset$ follows from the following stronger claim.
\begin{proposition}
\label{prop:4.7}
Let $\gamma$ and $(\omega ,c)$ satisfy \eqref{eq:4.1}. If a sequence $
\thickmuskip=3mu \{\varphi_n\}\subset X_{\omega,c}$ satisfies
\begin{align}
\label{eq:4.8}
\cS_{\omega ,c}(\varphi_n ) \to\mu (\omega ,c)~\text{and}~\cK_{\omega ,c}(\varphi_n ) \to 0~\text{as}~n\to\infty ,
\end{align}
then there exist a sequence $\{ y_n\}\subset\R$ and $v\in\scM_{\omega ,c}$ such that $\{ \varphi_n(\cdot -y_n)\}$ has a subsequence that converges to $v$ strongly in $X_{\omega ,c}$.
\end{proposition}
\begin{remark}
\label{rem:4.8}
If we only prove that $\scM_{\omega ,c}\neq\emptyset$, we may assume that $\cK_{\omega ,c}(\varphi_n)=0$ for all $n\in\N$. However, when one studies stability problems around the solitons, it is essential to consider the minimizing sequence $\{ \varphi_n\}$ satisfying $\cK_{\omega ,c}(\varphi_n)\neq 0$; see \cite{CO06, H19} or Section \ref{sec:6}.
\end{remark}
\begin{proof}
{\bf Step 1.} $\{ \varphi_n\}$ is bounded in $X_{\omega ,c}$. If $\omega >c^2/4$, this follows from (\ref{eq:4.7}) and Lemma \ref{lem:4.2} (i). If $c=2\sqrt{\omega}$, from (\ref{eq:4.5}) and \eqref{eq:4.7} we obtain that
\begin{align*}
\sup_{n\in\N}\|\varphi_n\|_{L^6}^6,~
\sup_{n\in\N}\|\del_x\l(e^{-i\frac{cx}{2}}\varphi_n\r)\|_{L^2}^2 <\infty .
\end{align*}
Since we have
\begin{align}
\label{eq:4.9}
\cK_{\omega ,c}(\varphi_n )&=\cL_{\omega ,c} (\varphi_n )
+\frac{c}{2}\| \varphi_n\|_{L^4}^4 -\frac{3}{16}\gamma \| \varphi_n\|_{L^6}^6,
\end{align}
we deduce that $\{ \varphi_n\}$ is also bounded in $L^4(\R)$.\\
{\bf Step 2.} $\limsup_{n\to\infty}\| \varphi_n\|_{L^6} >0$. Suppose that $\lim_{n\to\infty}\| \varphi_n\|_{L^6}=0$. If $\omega >c^2/4$, by the boundedness of $\{ \varphi_n\}$ in $L^2(\R)$ we have
\begin{align*}
\| \varphi_n\|_{L^4}^4 \leq\|\varphi_n\|_{L^2}\|\varphi_n\|_{L^6}^3 \underset{n\to\infty}{\longrightarrow} 0.
\end{align*}
From (\ref{eq:4.9}) we deduce that $\cL_{\omega ,c} (\varphi_n )\to 0$. By (\ref{eq:4.7}), we have $\cS_{\omega ,c}(\varphi_n) \to 0$, but this gives a contradiction with $\mu (\omega ,c)>0$. If $c=2\sqrt{\omega}$, from (\ref{eq:4.9}) we obtain that
\begin{align*}
\cL_{\omega ,c}(\varphi_n ), ~\|\varphi_n\|_{L^4}^4 \underset{n\to\infty}{\longrightarrow} 0,
\end{align*}
which yields $\cS_{\omega ,c}(\varphi_n) \to 0$ again. This gives a contradiction.\\
{\bf Step 3.} By Step 1, Step 2 and Lemma \ref{lem:4.5}, there exist $\{ y_n\}\subset\R$ and $v \in X_{\omega ,c}\setminus\{ 0\}$ such that a subsequence of $\{\varphi(\cdot -y_n)\}$ (we denote it by $\{ v_n\}$) converges to $v$ weakly in $X_{\omega ,c}$. Taking a subsequence if necessary, we have $v_n \to v$ a.e. in $\R$. 
By applying Lemma \ref{lem:4.6}, we have
\begin{align}
\label{eq:4.10}
&\cK_{\omega ,c}(v_n)-\cK_{\omega ,c}(v_n-v)-\cK_{\omega ,c}(v)\longrightarrow 0,\\
\label{eq:4.11}
&\cI_{\omega ,c}(v_n)-\cI_{\omega ,c}(v_n-v)-\cI_{\omega ,c}(v)\longrightarrow 0,
\end{align}
as $n\to\infty$.\\
{\bf Step 4.} $\cK_{\omega ,c}(v)\leq 0$. Suppose that $\cK_{\omega ,c}(v) >0$. By $\cK_{\omega ,c}(v_n) \to 0$ and \eqref{eq:4.10}, we have 
\begin{align*}
\cK_{\omega ,c}(v_n -v) \to -\cK_{\omega ,c}(v) <0.
\end{align*}
This implies that $\cK_{\omega ,c}(v_n -v)<0$ for large $n\in\N$. Applying Lemma \ref{lem:4.2} (iii), we have $\mu(\omega ,c)<\cI_{\omega ,c}(v_n -v)$ for large $n\in\N$. By \eqref{eq:4.8}
we have $\cI_{\omega ,c}(v_n) \to\mu (\omega ,c)$. Combined with (\ref{eq:4.11}), we have
\begin{align*}
 \cI_{\omega ,c}(v) =\lim_{n\to\infty}\l\{\cI_{\omega ,c}(v_n) -\cI_{\omega ,c}(v_n-v)\r\}
 \leq \mu(\omega ,c) -\mu (\omega ,c) =0,
 \end{align*}
 which yields that $v=0$. This is a contradiction.\\
{\bf Step 5.} By Step 4, Lemma \ref{lem:4.2} (iii), and the weakly lower semicontinuity of $\cI_{\omega ,c}$, we have
\begin{align*}
\mu (\omega ,c) \leq \cI_{\omega ,c} (v) \leq \liminf_{n\to\infty} \cI_{\omega ,c} (v_n) =\mu (\omega ,c).
\end{align*}
Thus we have $\cI_{\omega ,c}(v)=\mu (\omega ,c)$. By Step 4 and Lemma \ref{lem:4.2} (iii), we have $\cK_{\omega ,c}(v) =0$. Therefore $v\in\scM_{\omega ,c}$. By (\ref{eq:4.11}) and $\cI_{\omega ,c}(v)=\mu (\omega ,c)$, we have $\cI_{\omega ,c}(v_n -v) \to 0$, which yields that $v_n\to v$ strongly in $X_{\omega ,c}$. This completes the proof.
\end{proof}
\section{A two-parameter family of potential wells}
\label{sec:5}
In this section we prove Theorem \ref{thm:1.7}. We recall the following subsets of the energy space: 
\begin{align*}
\scA_{\omega ,c}=& \l\{ \varphi\in H^1(\R ):  \cS_{\omega, c}(\varphi ) < d(\omega ,c)\r\}, \\ 
\scA_{\omega, c}^+ =&\l\{ \varphi\in\scA_{\omega, c} :\cK_{\omega, c}(\varphi) \geq 0 \r\},\\
\scA_{\omega, c}^- =&\l\{ \varphi\in\scA_{\omega, c} :\cK_{\omega ,c}(\varphi) < 0\r\}.
\end{align*}
First we prove that $\scA_{\omega ,c}^{\pm}$ is invariant under the flow of \eqref{ME}. 
\begin{lemma}
\label{lem:5.1}
Let $b \geq -3/16$ and $(\omega,c)$ satisfy \eqref{eq:4.1}. Then, each of $\scA_{\omega ,c}^{+}$ and 
$\scA_{\omega ,c}^{-}$ is invariant under the flow of \eqref{ME}. If the initial data $v_0 \in\scA_{\omega ,c}^+$, then the corresponding solution is global, and satisfies the following uniform estimate\textup{:}
\begin{align}
\label{eq:5.1}
\| \del_x v\|_{L^{\infty}(\R ,L^2)}^2 \leq 8\cS_{\omega ,c}(v_0)+\frac{c^2}{2}\cM (v_0).
\end{align}
\end{lemma}
\begin{proof}
Assume that $v_0\in\scA_{\omega ,c}^+$. Let $v\in C((-T_{\rm min},T_{\rm max}) , H^1 (\R) )$ be a maximal solution of \eqref{ME} with $v(0)=v_0$. If $\cK_{\omega ,c}(v_0)=0$, by Proposition \ref{prop:4.1}, we have $v_0 =0$. By uniqueness we have $v(t)=0$ for all $t\in\R$. Consider the case $\cK_{\omega ,c}(v_0)>0$. If there exists $t_{*}\in (-T_{\rm min},T_{\rm max})$ such that $\cK_{\omega ,c}(v(t_{*}))=0$, the above argument gives that $v\equiv0$, which is a contradiction. Since the function $t\mapsto\cK_{\omega ,c}(v(t))$ is continuous, we deduce that $\cK_{\omega ,c}(v(t))>0$ for all $t\in (-T_{\rm min},T_{\rm max})$. This implies that $\scA_{\omega ,c}^+$ is invariant under the flow of \eqref{ME}. Similarly one can prove that $\scA_{\omega ,c}^-$ is also invariant.

Next we prove that the initial data $v_0\in\scA_{\omega ,c}^+$ generates global and bounded solutions. By (\ref{eq:4.7}) and $v(t)\in\scA_{\omega ,c}^+$, we obtain that
\begin{align*}
\cS_{\omega ,c}(v_0)&=\cS_{\omega ,c}(v(t))\\
&=\frac{1}{4}\cK_{\omega ,c}(v(t))+\frac{1}{4}\cL_{\omega ,c}(v(t)) 
+\frac{\gamma}{64}\| v(t)\|_{L^6}^6\\
&\geq\frac{1}{4}\cL_{\omega ,c}(v(t))\\
&\geq \frac{1}{4}\l\|\del_x\l( e^{-i\frac{cx}{2}}v(t)\r) \r\|_{L^2}^2
\end{align*}
for all $t\in (-T_{\rm min},T_{\rm max})$. This implies that $T_{\rm min}=T_{\rm max}=\infty$. 
Moreover, we have
\begin{align*}
\| \del_x v(t)\|_{L^2}^2&\leq\l( \l\| \del_x v(t) -\frac{c}{2}iv(t)\r\|_{L^2}+\frac{|c|}{2}\| v(t)\|_{L^2}\r)^2 \\
&\leq 2\l\|\del_x\l( e^{-i\frac{cx}{2}}v(t)\r) \r\|_{L^2}^2+\frac{c^2}{2}\cM (v_0) \\
&\leq 8\cS_{\omega ,c}(v_0)+\frac{c^2}{2}\cM (v_0)
\end{align*}
for all $t\in\R$. This completes the proof.
\end{proof}
We are now in a position to complete the proof of Theorem \ref{thm:1.7}. For convenience we often use the notation $\mu :=\sqrt{\omega}$ in the proof.
\begin{proof}[Proof of Theorem \ref{thm:1.7}]
First we note that
\begin{align}
\label{eq:5.2}
\begin{array}{ll}
\ds\text{if}~b\geq 0,& \ds \max_{-1 <s\leq 1}d(1,2s)=d(1,2s^*),\\[5pt]
\ds\text{if}~-3/16<b\leq 0,& \ds \max_{-1 <s\leq 1}d(1,2s)=d(1,2),
\end{array}
\end{align}
which follows from Lemma \ref{lem:2.9}. From \eqref{eq:3.8} and Proposition \ref{prop:2.6} we have
\begin{align}
\label{eq:5.3}
2d(1,2s) =M(\phi_{1,2s})+sP(\phi_{1,2s}).
\end{align}
Therefore, from the definition of $M^*$, the relation \eqref{eq:5.2} is rewritten as
\begin{align}
\label{eq:5.4}
\max_{-1 <s\leq 1}2d(1,2s) = M^* (b)=M^*
\end{align}
for $b>-3/16$. 

(i) First we prove the claim on the set above the mass threshold $M^*$.  Assume by contradiction that there exists $\varphi\in H^1(\R)$ such that $\cM(\varphi)>M^*$ and $\varphi\in\scA_s^+ \cap\scA_s^-$ for some $s\in (-1,1]$. Then, there exist $\mu_1, \mu_2 >0$ such that
\begin{align*}
&\cS_{\mu_i^2,2s\mu_i}(\varphi) <d(\mu_i^2, 2s\mu_i)\quad \text{for}~i=1,2,\\
&\cK_{\mu_1^2 ,2s\mu_1}(\varphi) <0, ~\cK_{\mu_2^2 ,2s\mu_2}(\varphi) >0.
\end{align*}
We may assume that $0<\mu_1<\mu_2$. Here we set the function $f_s:\R^+\to\R$ by 
\begin{align}
\label{eq:5.5}
f_s(\mu):=&\cS_{\mu^2,2s\mu}(\varphi) -d(\mu^2, 2s\mu)\\
=&\cE(\varphi)+\frac{\mu^2}{2}\Bigl( \cM(\varphi) -2d(1,2s)\Bigr) +s\mu\cP(\varphi).
\notag
\end{align}
From \eqref{eq:5.4} we have $\cM(\varphi) -2d(1,2s)>0$, which yields that the function $f_s$ is strictly convex. In particular $J_s:=\{\mu>0 : f_s(\mu)<0\}$ is an open interval which contains $\mu_1$ and $\mu_2$. From the explicit formula of  $\cK_{\mu^2,2s\mu}(\varphi)$ (see \eqref{eq:4.3}), there exists a unique $\mu_0\in (\mu_1, \mu_2)$ such that $\cK_{\mu_0^2,2s\mu_0}(\varphi)=0$. Since $\mu_0\in J_s$, in conclusion we deduce that there exists $\mu_0 >0$ such that
\begin{align*}
\cS_{\mu_0^2,2s\mu_0}(\varphi) <d(\mu_0^2, 2s\mu_0), ~\cK_{\mu_0^2,2s\mu_0}(\varphi)=0.
\end{align*}
However, from Proposition \ref{prop:4.1} this yields that $\varphi =0$, which is absurd. Therefore, $\scA^+_s$ and $\scA^-_s$ are mutually disjoint on $\{\varphi\in H^1(\R): \cM(\varphi)> M^*\}$.

Next we consider the case $\cM(\varphi)=M^*$. We only consider the case $b\geq 0$ since the case $-3/16<b<0$ is treated similarly. From \eqref{eq:5.4} we obtain that $\cM(\varphi) >2d(1,2s)$ for any $s\in (-1,1]$ but $s\neq s^*$. Hence, when $s\neq s^*$, one can use the argument above in the same way. When $s=s^*$, the function \eqref{eq:5.5} in this case is equal to 
\begin{align*}
f_{s^*}(\mu)=&\cS_{\mu^2,2s^*\mu}(\varphi) -d(\mu^2, 2s^*\mu)
=\cE(\varphi) +s^*\mu\cP(\varphi).
\end{align*}
From this formula, we deduce that $J_{s^*}$ is an open interval if it is not empty. Hence the argument above still holds in this case. This completes the proof of (i).

(ii) Assume that $\cM (\varphi)<M^*$, or $\cM (\varphi)=M^*$ and $\cP (\varphi)<0$. We note that for any $\varphi\in H^1(\R)\setminus\{ 0\}$ there exists large $\omega >0$ such that
\begin{align}
\label{eq:5.6}
\cK_{\omega ,2s\sqrt{\omega}}(\varphi) 
&= \| \del_x \varphi\|_{L^2}^2 + \omega\| \varphi\|_{L^2}^2 
\\
&\quad +s \sqrt{\omega}\l( 2\rbra[i\del_x \varphi, \varphi] +\| \varphi\|_{L^4}^4 \r) -\frac{3}{16}\gamma\| \varphi\|_{L^6}^6> 0,
\notag
\end{align}
where $\omega$ depends on $s$ and $\varphi$. We also note that 
\begin{align}
\label{eq:5.7}
&\cS_{\omega ,2s\sqrt{\omega}}(\varphi) < d(\omega ,2s\sqrt{\omega}) 
\Leftrightarrow  \cE (\varphi)+s\sqrt{\omega} \cP (\varphi) < \frac{ \omega}{2} \bigl( 2d(1,2s)- \cM (\varphi) \bigr)  
\end{align}
for any $s\in (-1,1]$. When $b\geq 0$, if we set $s=s^*$, the last inequality in \eqref{eq:5.7} holds for large $\omega>0$ from \eqref{eq:5.2} and \eqref{eq:5.4}. Combined with (\ref{eq:5.6}), we deduce that $\varphi\in\scA_{s^*}^+$. When $-3/16<b\leq 0$, if we set $s=1$, $\varphi\in\scA_{1}^+$ is proved in the same way.

(iii)(a) From the definition of $\cS_{\omega ,c}$, we have
\begin{align*}
\cS_{\mu^2,2\mu}(e^{i\mu x}\psi )&=\frac{1}{2}\cL_{\mu^2,2\mu}(e^{i\mu x}\psi) 
+\frac{\mu}{4}\| \psi \|_{L^4}^4-\frac{\gamma}{32}\| \psi\|_{L^6}^6\\
&=\frac{1}{2}\| \del_x \psi\|_{L^2}^2 +\frac{\mu}{4}\| \psi\|_{L^4}^4-\frac{\gamma}{32}\| \psi \|_{L^6}^6.
\end{align*}
Since $d(\mu^2 ,2\mu)=\mu^2 d(1,2)$ and $d(1,2)>0$, we deduce that 
\begin{align*}
\cS_{\mu^2,2\mu}(e^{i\mu x}\psi ) < d(\mu^2 ,2\mu)
\end{align*}
for large $\mu>0$. Similarly, we have
\begin{align*}
\cK_{\mu^2,2\mu}(e^{i\mu x}\psi) &=\cL_{\mu^2,2\mu}(e^{i\mu x}\psi)
+\mu \| \psi\|_{L^4}^4-\frac{3}{16}\gamma\| \psi\|_{L^6}^6\\
&=\|\del_x\psi \|_{L^2}^2 +\mu\| \psi \|_{L^4}^4-\frac{3}{16}\gamma\| \psi\|_{L^6}^6 > 0
\end{align*}
for large $\mu>0$. This yields that $e^{i\mu x}\psi\in\scA^+_1$.

(b) First we note that
\begin{align*}
\cS_{\mu^2,2s\mu}(e^{is\mu x}\psi )&=\frac{1}{2}\| \del_x \psi\|_{L^2}^2 
+\frac{\mu^2}{2} \l( 1-s^2\r) \| \psi\|_{L^2}^2+ \frac{s\mu}{4}\| \psi\|_{L^4}^4-\frac{\gamma}{32}\| \psi \|_{L^6}^6,
\\
\cK_{\mu^2,2s\mu}(e^{is\mu x}\psi )&=\| \del_x \psi\|_{L^2}^2 
+\mu^2 \l( 1-s^2\r) \| \psi\|_{L^2}^2+ s\mu\| \psi\|_{L^4}^4-\frac{3}{16}\gamma\| \psi \|_{L^6}^6
\end{align*}
for any $s\in (-1,1]$. We fix large $\mu >0$ such that 
\begin{align*}
\cS_{\mu^2,-2\mu}(e^{-i\mu x}\psi )&=\frac{1}{2}\| \del_x \psi\|_{L^2}^2 
-\frac{\mu}{4}\| \psi\|_{L^4}^4-\frac{\gamma}{32}\| \psi \|_{L^6}^6<0,\\
\cK_{\mu^2,-2\mu}(e^{-i\mu x}\psi )&=\| \del_x \psi\|_{L^2}^2 
-\mu\| \psi\|_{L^4}^4-\frac{3}{16}\gamma\| \psi \|_{L^6}^6<0.
\end{align*}
We note that  
\begin{align*}
\cS_{\mu^2,-2\mu}(e^{is\mu x}\psi )&=\lim_{s\downarrow -1}\cS_{\mu^2,2s\mu}(e^{is\mu x}\psi ),\\
\cK_{\mu^2,-2\mu}(e^{-i\mu x}\psi )&=\lim_{s\downarrow -1}\cK_{\mu^2,2s\mu}(e^{is\mu x}\psi ),
\end{align*}
and $\lim_{s\downarrow -1}d(1,2s)=0$. Therefore, there exists small $\eps >0$ such that for any $s\in (-1,-1+\eps)$ we have
\begin{align*}
\cS_{\mu^2,2s\mu}(e^{is\mu x}\psi ) <d(\mu^2,2s\mu),~
\cK_{\mu^2,2s\mu}(e^{is\mu x}\psi )<0.
\end{align*}
This yields that $e^{is\mu x}\psi\in\scA^-_{s}$ for $s\in (-1,-1+\eps)$.

(iv) Assume $\cE(\varphi)<0$. We note that the functional $\cK_{\omega ,c}$ is rewritten as
\begin{align}
\label{eq:5.8}
\cK_{\omega ,c}(\varphi) =6\cE(\varphi) -2\| \del_x\varphi\|_{L^2}^2+\omega\|\varphi\|_{L^2}^2
+c\rbra[i\del_x\varphi, \varphi] +\frac{c}{2}\| \varphi\|_{L^4}^4 .
\end{align}
From \eqref{eq:5.7} and this formula, we deduce that for each $s\in (-1,1]$ there exists small $\omega >0$ such that 
\begin{align*}
\cS_{\omega ,2s\sqrt{\omega}}(\varphi) < d(\omega ,2s\sqrt{\omega}) , ~
\cK_{\omega ,2s\sqrt{\omega}}(\varphi)<0.
\end{align*}
This yields that $\varphi\in\bigcap_{-1<s\leq 1}\scA^-_s$. If we assume further that $\cM (\varphi)\geq M^*$, it follows from (i) that $\varphi\notin\bigcup_{-1< s\leq 1}\scA_s^+$. 

(v) Assume by contradiction that $\varphi\in\bigcup_{0\leq s\leq 1}\scA_s$ under the assumption
\begin{align*}
\cE (\varphi) \geq0, \cM(\varphi) \geq M^*~\text{and}~\cP (\varphi)\geq 0.
\end{align*}
Then, there exist $s_0\in [0,1]$ and $\omega_0 >0$ such that $\cS_{\omega_0 ,2s_0\sqrt{\omega_0}}(\varphi) < d(\omega_0 ,2s_0\sqrt{\omega_0})$. This is equivalent that
\begin{align*}
 \cE (\varphi)+\frac{ \omega_0}{2} \bigl( \cM (\varphi) -2d(1,2s_0) \bigr) +s_0\sqrt{\omega_0} \cP (\varphi) < 0.
\end{align*}
But this is absurd, since $\cM (\varphi) -2d(1,2s_0)\geq 0$ from \eqref{eq:5.4}. 

In the same way one can prove that 
\begin{align*}
\cE (\varphi) \geq0, \cM(\varphi) \geq M^*~\text{and}~\cP (\varphi)\leq 0 \Longrightarrow\varphi\notin
\bigcup_{-1< s\leq 0}\scA_s.
\end{align*}

(vi)(a) Let $b\geq 0$. Assume that $\cM(\varphi)=M^*$, $\cE(\varphi)\leq 0$ and $\cP(\varphi)\leq 0$ except for the case $\cE(\varphi)=\cP(\varphi)=0$. We note that the function $f_{s^*}$ defined by \eqref{eq:5.5} has the following formula:
\begin{align*}
f_{s^*}(\mu) =\cE(\varphi) +s^*\mu\cP(\varphi).
\end{align*}
From the assumption we note that $\{ \mu>0:f_{s^*}(\mu)<0\}=\R^+$. 
From the formulae \eqref{eq:5.6} and \eqref{eq:5.8}, we deduce that
\begin{align}
\label{eq:5.9}
\begin{array}{ll}
\cK_{\mu^2 ,2s^*\mu}(\varphi) >0 ~&\text{for large}~\mu>0, \\
\cK_{\mu^2 ,2s^*\mu}(\varphi) <0 ~&\text{for small}~\mu>0.
\end{array}
\end{align}
Therefore, there exists $\mu_0 >0$ such that
\begin{align*}
\cS_{\mu_0^2,2s^*\mu_0}(\varphi) <d(\mu_0^2, 2s^*\mu_0), ~\cK_{\mu_0^2,2s^*\mu_0}(\varphi)=0.
\end{align*}
However, from Proposition \ref{prop:4.1} this yields that $\varphi =0$, which is absurd. 

Now we consider the case $\cM(\varphi)=M^*$ and $\cE(\varphi)=\cP(\varphi)=0$. In this case we have $f_{s^*}\equiv 0$, which is equivalent that $\cS_{\mu^2,2s^*\mu}(\varphi) =d(\mu^2, 2s^*\mu)$ for any $\mu>0$. Since the statement \eqref{eq:5.9} still holds in this case, we deduce that there exists a unique $\mu_0 >0$ such that
\begin{align*}
\cS_{\mu_0^2,2s^*\mu_0}(\varphi) =d(\mu_0^2, 2s^*\mu_0), ~\cK_{\mu_0^2,2s^*\mu_0}(\varphi)=0.
\end{align*}
From Proposition \ref{prop:4.1} again, there exist $\theta , y\in\R$ such that $\varphi =e^{i\theta}\varphi_{\mu_0^2 ,2s^*\mu_0 }(\cdot -y)$. This completes the proof of (vi-a).

(b) In the same way as in the proof of (vi-a), we deduce that there exist no $\varphi\in H^1(\R)$ such that $\cM(\varphi)=M^*$, $\cE(\varphi)\leq 0$ and $\cP(\varphi)\leq 0$ except for the case $\cE(\varphi)=\cP(\varphi)=0$. We consider the case $\cM(\varphi)=M^*$ and $\cE(\varphi)=\cP(\varphi)=0$. Similarly, one can prove that there exists a unique $\mu_0 >0$ such that
\begin{align*}
\cS_{\mu_0^2,2\mu_0}(\varphi) =d(\mu_0^2, 2\mu_0), ~\cK_{\mu_0^2,2\mu_0}(\varphi)=0.
\end{align*}
From Proposition \ref{prop:4.1}, there exist $\theta , y\in\R$ such that $\varphi =e^{i\theta}\varphi_{\mu_0^2 ,2\mu_0 }(\cdot -y)$. But this is a contradiction, since $\cP(\varphi_{1,2})>0$ from Proposition \ref{prop:1.3}. This completes the proof. 
\end{proof}
In the case $b=-3/16$, we have the following result.
\begin{proposition}
\label{prop:5.2}
Let $b=-3/16$. Then we have
\begin{align*}
\bigcup_{-1< s<0}\scA_s =\bigcup_{-1< s<0}\scA_s^+=H^1(\R) .
\end{align*}
\end{proposition}
\begin{proof}
Given $\varphi\in H^1(\R)$. From \eqref{eq:5.3} and Proposition \ref{prop:2.4} we have
\begin{align*}
d(1,2s) =\frac{1-s^2}{3}M(\phi_{1,2s})
\end{align*}
for $s\in(-1,0)$. It follows from Proposition \ref{prop:2.2} that $d(1,2s)\to\infty$ as $s\to 0-$. Hence, for $\varphi\in H^1(\R)$ one can take $s_0\in (-1,0)$ such that
\begin{align}
\label{eq:5.10}
2\,d(1,2s_0)- \cM (\varphi)  >0.
\end{align}
We note that
\begin{align*}
\cS_{\omega ,2s_0\sqrt{\omega}}(\varphi) < d(\omega ,2s_0\sqrt{\omega})
\Leftrightarrow  \cE (\varphi)+{s_0\sqrt{\omega}}\cP (\varphi)  
< \frac{\omega}{2} \l( 2\,d(1,2s_0)- \cM (\varphi) \r) . 
\end{align*}
By \eqref{eq:5.10} the last inequality holds for large $\omega >0$. Combined with \eqref{eq:5.6}, we deduce that $\varphi\in\scA_{s_0}^+$. This completes the proof.
\end{proof}
\section{Potential well theory for the Hamiltonian form}
\label{sec:6}
When $b\geq 0$ one can give a variational characterization of solitons to the equation \eqref{eq:1.1} on the Nehari manifold. Based on this variational characterization one can establish potential well theory for \eqref{eq:1.1} similarly as in the case of \eqref{ME}.
We recall that the equation \eqref{eq:1.1} has 
Hamiltonian structure \eqref{eq:1.3} which is useful when one studies problems of solitons and related topics. Hence it would be worthwhile to restate potential well theory for the Hamiltonian form. Here we only consider the case of \eqref{DNLS} for simplicity. After we organize potential well theory for \eqref{DNLS}, we give a proof of Theorem \ref{thm:1.12}.

Let us consider the case $b=0$. We define the functional by $K_{\omega ,c}(\varphi):=\l.\frac{d}{d\lambda}S_{\omega ,c}(\lambda\varphi)\r|_{\lambda =1}$, which has the following explicit formula:
\begin{align}
\label{eq:6.1}
K_{\omega ,c}(\varphi) =\| \del_x\varphi\|_{L^2}^2+\omega\|\varphi\|_{L^2}^2
+c\rbra[i\del_x\varphi, \varphi] -\rbra[ i|\varphi|^2\del_x\varphi, \varphi].
\end{align}
We consider the following subsets of the energy space: 
\begin{align*}
\scK_{\omega ,c}:=& \l\{ \varphi\in H^1(\R ):  S_{\omega, c}(\varphi ) < S_{\omega ,c}(\varphi_{\omega, c})\r\}, \\ 
\scK_{\omega, c}^+ :=&\l\{ \varphi\in\scK_{\omega, c} :K_{\omega, c}(\varphi) \geq 0 \r\},
\scK_{\omega, c}^- :=\l\{ \varphi\in\scK_{\omega, c} :K_{\omega ,c}(\varphi) < 0\r\},\\
\scK_{s}:=& \bigcup_{ \substack{\omega>0 } } \scK_{\omega ,2s\sqrt{\omega}},~
\scK_{s}^{\pm}:= \bigcup_{ \substack{\omega>0 } } \scK_{\omega ,2s\sqrt{\omega}}^{\pm}
\quad\text{for}~s\in (-1,1].
\end{align*}
Now we can rewrite Theorem \ref{thm:1.7} as the following claim for \eqref{DNLS}:
\begin{theorem}
\label{thm:6.1}
Let $(\omega ,c)$ satisfy $-2\sqrt{\omega}<c\leq 2\sqrt{\omega}$. Then, each of $\scK_{\omega ,c}^{+}$ and 
$\scK_{\omega ,c}^{-}$ is invariant under the flow of \eqref{DNLS}. If $u_0 \in\scK_{\omega ,c}^+$, then the $H^1(\R)$-solution $u$ of \eqref{DNLS} with $u(0)=u_0$ exists globally both forward and backward in time, and satisfies the following uniform estimate\textup{:}
\begin{align}
\label{eq:6.2}
\| \del_x v\|_{L^{\infty}(\R ,L^2)}^2 \leq 8S_{\omega ,c}(u_0)+\frac{c^2}{2} M (u_0).
\end{align}
Moreover the following statements hold\textup{:}
\begin{enumerate}[\rm (i)]
\setlength{\itemsep}{3pt}
\item For each $s\in (-1,1]$, $\scK^+_{s}$ and $\scK^-_s$ have no elements in common on the set $\{ \varphi\in H^1(\R) : M(\varphi ) \geq 4\pi\}$. 

\item  If $M (\varphi)<4\pi$, or $M (\varphi)=4\pi$ and $P (\varphi)<0$, then $\varphi\in\scK^+_{1}$.

\item For given $\psi\in H^1(\R)\setminus\{ 0\}$ the following properties hold\textup{:}
\begin{enumerate}[\rm (a)]
\item There exists $\mu_0 =\mu_0 (\psi)>0$ such that if $\mu\geq\mu_0$, then $e^{i\mu x}\psi\in \scK^+_1$.

\item There exist $\eps\in (0,1)$ and large $\mu >0$ such that $e^{-i(1-\eps )\mu x}\psi\in \scK^-_{-(1-\eps)}$, where $\eps$ and $\mu$ depend on $\psi$.
\end{enumerate}

\item Assume $E(\varphi)<0$. Then $\varphi\in\bigcap_{-1<s\leq 1}\scK^-_s$. 
In particular, if $M(\varphi)\geq 4\pi$, then $\varphi\not\in\bigcup_{-1< s\leq 1}\scK_s^+$.

\item Assume $E(\varphi) \geq 0$ and $M (\varphi) \geq 4\pi$. 
If $P(\varphi) \geq 0\,(\text{resp.}\,P(\varphi)\leq 0)$, then $\varphi\not\in\bigcup_{0\leq s\leq 1}\scK_s\,(\text{resp.}\,\varphi\not\in\bigcup_{-1<s\leq 0}\scK_s)$. In particular, if $P(\varphi)=0$, then $\varphi\not\in\bigcup_{-1< s\leq 1}\scK_s$.

\item Assume $M (\varphi)=4\pi$. $E(\varphi)=P(\varphi)=0$ if and only if there exist $\theta , y\in\R$ and $\omega >0$ such that $\varphi =e^{i\theta}\phi_{\omega ,2\sqrt{\omega} }(\cdot -y)$. Moreover, there exists no $\varphi\in H^1(\R)$ such that $E(\varphi)<0$ and $P(\varphi)\leq 0$, or $E(\varphi)\leq 0$ and $P(\varphi) <0$.

\end{enumerate}
\end{theorem}
Theorem \ref{thm:6.1} characterizes $4\pi$-mass condition for \eqref{DNLS} from the viewpoint of potential well theory. We note that algebraic solitons give the boundary of both $\scK^+_1$ and $\scK^-_1$. This is a notable property of algebraic solitons, which is analogous to the one of standing  waves for \eqref{NLS}.

We also note that Theorem \ref{thm:1.12} gives another interesting property of algebraic solitons. For the proof of Theorem \ref{thm:1.12}, we use the following lemma on the concentration compactness, which is corresponding to Proposition \ref{prop:4.7}.
\begin{proposition}[\cite{CO06, FHI17}]
\label{prop:6.2}
Let $(\omega ,c)$ satisfy $-2\sqrt{\omega}<c\leq 2\sqrt{\omega}$. Let $X_{\omega ,c}$ be defined by \eqref{eq:4.2}. If a sequence $
\thickmuskip=3mu \{\varphi_n\}\subset X_{\omega,c}$ satisfies
\begin{align*}
S_{\omega ,c}(\varphi_n ) \to d(\omega ,c)~\text{and}~K_{\omega ,c}(\varphi_n ) \to 0~\text{as}~n\to\infty ,
\end{align*}
then there exist a sequence $\{ y_n\}\subset\R$ and $\theta_0,y_0\in\R$ such that $\{ \varphi_n(\cdot -y_n)\}$ has a subsequence that converges to $e^{i\theta_0}\phi_{\omega ,c}(\cdot -y_0)$ strongly in $X_{\omega ,c}$.
\end{proposition}
\begin{proof}[Proof of Theorem \ref{thm:1.12}]
It is enough to prove that for any sequence ${t_n}$ such that $t_n\to T^*$, there exist a subsequence (denote it again by $\{ t_n\}$), and sequences $y(t_n)$ and $\theta(t_n)$ such that 
\begin{align}
\label{eq:6.3}
e^{i\theta(t_n)}u_{\lambda (t_n)}(t_n, \cdot +y(t_n))\to \phi_{1,2}~\text{in}~H^1(\R)~\text{as}~t_n\to T^*,
\end{align}
where $u_{\lambda(t)}(t,x) :=\lambda(t)^{1/2}u(t, \lambda(t)x)$.  
Since $\lambda (t_n)\to 0$ as $t_n\to T^*$, we obtain that
\begin{align*}
E(u_{\lambda(t_n)}(t_n))&=\lambda (t_n)^2E(u_0)  \to 0,\\
P(u_{\lambda(t_n)}(t_n))&=\lambda (t_n)P(u_0) \to 0,
\end{align*}
and $M(u_{\lambda(t_n)}(t_n))=M(u_0)=4\pi$ for any $n\in\N$. Hence we have
\begin{align}
\label{eq:6.4}
S_{1,2}(u_{\lambda (t_n)}(t_n))&=E(u_{\lambda(t_n)}(t_n))+\frac{1}{2}M(u_{\lambda(t_n)}(t_n))+P(u_{\lambda(t_n)}(t_n))\\
&\underset{t_n\to T^*}{\longrightarrow}\frac{1}{2}\cdot 4\pi =d(1,2).
\notag
\end{align}
We note that the functional $K_{\omega ,c}$ is rewritten as
\begin{align*}
K_{\omega ,c}(\varphi) =-\|\del_x\varphi\|_{L^2}^2 +4E(\varphi)+\omega M(\varphi)+cP(\varphi).
\end{align*}
Since $E(\phi_{1,2})=P(\phi_{1,2})=0$, we have
\begin{align*}
0=K_{1 ,2}(\phi_{1,2})=-\|\del_x\phi_{1,2}\|_{L^2}^2 +M(\phi_{1,2}).
\end{align*}
Therefore, we deduce that
\begin{align}
\label{eq:6.5}
K_{1,2}(u_{\lambda (t_n)}(t_n))&=-\lambda(t_n)^2\| \del_xu(t_n) \|_{L^2}^2+
4E(u_{\lambda(t_n)}(t_n))\\
&\quad +M(u_{\lambda(t_n)}(t_n) )+2P(u_{\lambda(t_n)}(t_n))
\notag\\
&\underset{t_n\to T^*}{\longrightarrow} -\|\del_x\phi_{1,2}\|_{L^2}^2 +M(\phi_{1,2})=0.
\notag
\end{align}
Therefore, by \eqref{eq:6.4}, \eqref{eq:6.5} and Proposition \ref{prop:6.2}, 
there exist a subsequence of $\{ t_n\}$, $\theta (t_n)$ and $y(t_n)$ such that 
\begin{align*}
g(t_n):=e^{i\theta(t_n)}u_{\lambda (t_n)}(t_n, \cdot +y(t_n))\underset{t_n\to T^*}{\longrightarrow} \phi_{1,2}~\text{in}~X_{1,2}.
\end{align*}
In particular we have
\begin{align}
\label{eq:6.6}
e^{-ix}g(t_n)\to e^{-ix}\phi_{1,2}~\text{in}~\dot{H}^1(\R),
\\
\label{eq:6.7}
e^{-ix}g(t_n)\wto e^{-ix}\phi_{1,2}~\text{weakly in}~L^2(\R).
\end{align}
From the mass conservation, we note that
\begin{align}
\label{eq:6.8}
M(g(t_n))=M(u_{\lambda_n}(t_n))=4\pi=M(\phi_{1,2})
\end{align}
for any $n\in\N$. Combined with \eqref{eq:6.7}, we obtain that
\begin{align}
\label{eq:6.9}
e^{-ix}g(t_n)\to e^{-ix}\phi_{1,2}~\text{strongly in}~L^2(\R).
\end{align}
From \eqref{eq:6.6} and \eqref{eq:6.9}, we obtain \eqref{eq:6.3}. This completes the proof. 
\end{proof}

\section*{Acknowledgments}
The results of this paper were mostly obtained when the author was a PhD student at Waseda University. The author would like to thank his thesis adviser Tohru Ozawa for constant encouragements.
The author is also grateful to Noriyoshi Fukaya and Takahisa Inui for useful discussions, and to Nobu Kishimoto, Kenji Nakanishi, and Yoshio Tsutsumi for helpful comments. This work was supported by JSPS KAKENHI Grant Numbers JP17J05828, JP19J01504, and Top Global University Project, Waseda University.


\begin{thebibliography}{99}

\bibitem{BFV14} J. Bellazzini, R.L. Frank, N. Visciglia, Maximizers for Gagliardo--Nirenberg inequalities and related non-local problems, Math. Ann. {\bf 360} (2014), 653--673.

\bibitem{BeL83}
H. Berestycki, P.-L. Lions, Nonlinear scalar field equations. I. Existence of a ground state, 
Arch. Rational Mech. Anal. {\bf 82} (1983), 313--345.

\bibitem{BL01} H. Biagioni, F. Linares, Ill-posedness for the derivative Schr\"{o}dinger and generalized Benjamin-Ono equations, Trans. Amer. Math. Soc. {\bf 353} (2001), 3649--3659.

\bibitem{BL83} H. Br\'{e}zis, E.H. Lieb, A relation between pointwise convergence of functions and convergence of functionals, 
Proc. Amer. Math. Soc. {\bf 88} (1983), 486--490.

\bibitem{C03}  T. Cazenave, {\it Semilinear Schr\"{o}dinger Equations}, Courant Lecture Notes in Math. vol.{\bf 10}, Amer. Math. Soc., 2003.

\bibitem{CW92} T. Cazenave, F.B. Weissler, Rapidly decaying solutions of the nonlinear Schr\"{o}dinger equation, Comm. Math. Phys. {\bf 147} (1992), 75--100. 

\bibitem{CO06} M. Colin, M. Ohta, Stability of solitary waves for derivative nonlinear Schr\"{o}dinger equation, Ann. Inst. H. Poincar\'{e} Anal. Non Lin\'eaire {\bf 23} (2006), 753--764.

\bibitem{CKSTT01} J. Colliander, M. Keel, G. Staffilani, H. Takaoka, T. Tao, Global well-posedness for Schr\"{o}dinger equations with derivative, SIAM J. Math. Anal. {\bf 33} (2001), 649--669.

\bibitem{CKSTT02} J. Colliander, M. Keel, G. Staffilani, H. Takaoka, T. Tao, A refined global well-posedness for Schr\"{o}dinger equations with derivative, SIAM J. Math. Anal. {\bf 34} (2002), 64--86.

\bibitem{D15} B. Dodson, Global well-posedness and scattering for the mass critical nonlinear Schr\"odinger equation with mass below the mass of the ground state, Adv. Math. {\bf 285} (2015), 
1589--1618. 

\bibitem{FHI17} N. Fukaya, M. Hayashi, T. Inui, A sufficient condition for global existence of solutions to a generalized derivative nonlinear Schr\"{o}dinger equation, Anal. PDE {\bf 10} (2017), 1149--1167.

\bibitem{GR07} I. Gradshteyn and I. Ryzhik, \textit{Table of Integrals, Series, and Products}, 7th edition, Elsevier/Academic Press, Amsterdam, 2007.



\bibitem{GW95} B. Guo, Y. Wu, 
Orbital stability of solitary waves for the nonlinear derivative Schr\"{o}dinger equation, 
J. Differential Equations {\bf 123} (1995), 35--55.

\bibitem{GHLN13} Z. Guo, N. Hayashi, Y. Lin, P. Naumkin, Modified scattering operator for the derivative nonlinear Schr\"{o}dinger equation, SIAM J. Math. Anal. {\bf 45} (2013), 3854--3871. 

\bibitem{GW17} 
Z. Guo, Y. Wu, Global well-posedness for the derivative nonlinear Schr\"{o}dinger equation in $H^{\frac{1}{2}}(\R)$, 
Discrete Contin. Dyn. Syst. {\bf 37} (2017), 257--264.

\bibitem{H18} M. Hayashi, Long-period limit of exact periodic traveling wave solutions for the derivative nonlinear Schr\"{o}dinger equation, Ann. Inst. H. Poincar\'{e} Anal. Non Lin\'eaire {\bf 36} (2019), 1331--1360.

\bibitem{H19} M. Hayashi, Stability of algebraic solitons for nonlinear Schr\"{o}dinger equations of derivative type: variational approach, preprint (arXiv:2011.08029). 

\bibitem{HO16} M. Hayashi, T. Ozawa, Well-posedness for a generalized derivative nonlinear Schr\"{o}dinger equation, J. Differential Equations {\bf 261} (2016), 5424--5445.

\bibitem{H93} N. Hayashi, The initial value problem for the derivative nonlinear Schr\"{o}dinger equation in the energy space,
Nonlinear Anal. {\bf 20} (1993), 823--833.

\bibitem{HO92} N. Hayashi, T. Ozawa, On the derivative nonlinear Schr\"{o}dinger equation,
Phys. D. {\bf 55} (1992), 14--36.

\bibitem{HO94a} N. Hayashi, T. Ozawa, Finite energy solutions of nonlinear Schr\"{o}dinger equations of derivative type, SIAM J. Math. Anal. {\bf 25} (1994), 1488--1503. 

\bibitem{HO94modi} N. Hayashi, T. Ozawa, Modified wave operators for the derivative nonlinear 
Schr\"{o}dinger equation, Math. Ann. {\bf 298} (1994), 557--576. 

\bibitem{JLPS18}
R. Jenkins, J. Liu, P. Perry, C. Sulem, Global existence for the derivative nonlinear Schr\"{o}dinger equation with arbitrary spectral singularities, preprint, arXiv:1804.01506.

\bibitem{JLPS18a}
R. Jenkins, J. Liu, P. Perry, C. Sulem, Soliton resolution for the derivative nonlinear Schr\"{o}dinger equation, Comm. Math. Phys. {\bf 363} (2018), 1003--1049. 

\bibitem{KN78} D. J. Kaup, A. C. Newell, An exact solution for a derivative nonlinear Schr\"{o}dinger equation, J. Math. Phys., {\bf 9} (1978), 789--801.

\bibitem{KW18} S. Kwon, Y. Wu, Orbital stability of solitary waves for derivative nonlinear Schr\"{o}dinger equation, J. Anal. Math. {\bf 135} (2018), 473--486; Erratum: see arXiv:1603.03745, last revised on Oct. 30, 2019.


\bibitem{Lie83} E.H. Lieb, On the lowest eigenvalue of the Laplacian for the intersection of two domains, 
Invent. Math. {\bf 74} (1983), 441--448.


\bibitem{MOMT76} K. Mio, T. Ogino, K. Minami, S. Takeda, Modified nonlinear Schr\"{o}dinger equation for Alfv\'{e}n waves propagating along the magnetic field in cold plasma, 
J. Phys. Soc. Japan {\bf 41} (1976), 265--271.

\bibitem{M76} E. Mj\o lhus, On the modulational instability of hydromagnetic waves parallel to the magnetic field, J. Plasma Phys. {\bf 16} (1976), 321--334.

\bibitem{NOW17} C. Ning, M. Ohta, Y. Wu, Instability of solitary wave solutions for derivative nonlinear Schr\"odinger equation in endpoint case, J. Differential Equations {\bf 262} (2017), 1671--1689.

\bibitem{O14} M. Ohta, Instability of solitary waves for nonlinear Schr\"{o}dinger equations of derivative type, SUT J. Math. {\bf 50} (2014), 399--415.

\bibitem{OT91} T. Ogawa, Y. Tsutsumi, Blow-up of $H^1$ solutions for the one-dimensional nonlinear Schr\"odinger equation with critical power nonlinearity, Proc. Amer. Math. Soc. {\bf 111} (1991), 487--496. 

\bibitem{Oz96} T. Ozawa, On the nonlinear Schr\"{o}dinger equations of derivative type, Indiana Univ. Math. J. {\bf 45} (1996), 137--163.

\bibitem{PS75}
L.E. Payne, D.H. Sattinger, Saddle points and instability of nonlinear hyperbolic equations, Israel J. Math. {\bf 22} (1975), 273--303. 

\bibitem{S83} J. Shatah, Stable standing waves of nonlinear Klein-Gordon equations, Comm. Math. Phys. {\bf 91} (1983), 313--327.


\bibitem{T99} H. Takaoka, Well-posedness for the one dimensional Schr\"{o}dinger equation with the derivative nonlinearity, Adv. Diff. Eq. {\bf 4} (1999), 561--580.

\bibitem{Tan04}  S.B. Tan, Blow-up solutions for mixed nonlinear Schr\"{o}dinger equations, Acta Math. Sin. {\bf 20} (2004), 115--124.

\bibitem{TF80} M. Tsutsumi, I. Fukuda, On solutions of the derivative nonlinear Schr\"{o}dinger equation. Existence and Uniqueness, 
Funkcial. Ekvac. {\bf 23} (1980), 259--277.

\bibitem{TF81} M. Tsutsumi, I. Fukuda, On solutions of the derivative nonlinear Schr\"{o}dinger equation. II, Funkcial. Ekvac. {\bf 24} (1981), 85--94.

\bibitem{W82} M. I. Weinstein, Nonlinear Schr\"{o}dinger equations and sharp interpolation estimates, Comm. Math. Phys. {\bf 87} (1982), 567--576.

\bibitem{W86} M. I. Weinstein, On the structure and formation of singularities in solutions to nonlinear dispersive evolution equations, Comm. Partial Differential Equations {\bf 11} (1986), 545--565.

\bibitem{Wu13} Y. Wu, Global well-posedness for the nonlinear Schr\"{o}dinger equation with derivative in energy space,
Anal. PDE {\bf 6} (2013), 1989--2002. 

\bibitem{Wu15} Y. Wu, Global well-posedness on the derivative nonlinear Schr\"{o}dinger equation,
Anal. PDE {\bf 8} (2015), 1101--1112. 

\end{thebibliography}
\end{document}